\documentclass[11pt]{article}

\usepackage{amsthm, amssymb, bm, amsmath}
\usepackage{soul, color,tikz-cd}

\usepackage[pagebackref,colorlinks=true,pdfpagemode=none,urlcolor=blue,linkcolor=blue,citecolor=blue]{hyperref}
\usepackage[authoryear,round]{natbib}

\newtheorem{theorem}{Theorem}[]

\newtheorem{proposition}[theorem]{Proposition}
\newtheorem{corollary}[theorem]{Corollary}
\newtheorem{definition}[theorem]{Definition}
\newtheorem{remark}[theorem]{Remark}

\newtheorem{conjecture}{Conjecture}[section]

\newcommand \Cm { \mathbb{C}}
\newcommand \Dm { \mathbb{D}}
\newcommand \Rm { \mathbb{R}}
\newcommand \Nm { \mathbb{N}}
\newcommand \Sm { \mathbb{S}}
\newcommand \Zm { \mathbb{Z}}

\renewcommand \L {{\cal L}}
\newcommand \D {{\cal D}}

\newcommand \bt { {\mathbf t}}

\newcommand\x {{\mathrm{x}}}

\newcommand \wtH {\widetilde{H}}
\newcommand \PsiDO {$\Psi$DO}

\newcommand \E { {\cal{E}}}
\newcommand \F { {\cal{F}}}
\newcommand \V { {\cal{V}}}
\newcommand \A { {\cal{A}}}
\newcommand \B { {\cal{B}}}
\newcommand \phg { {\textrm{phg}}}

\newcommand \two { {\mathrm {II}}}

\title{Non-standard Sobolev scales and the mapping properties of the X-ray transform on manifolds with strictly convex boundary}
\author{Fran\c{c}ois Monard\thanks{Department of Mathematics, University of California, Santa Cruz CA 95064. Email: fmonard@ucsc.edu.}}
\date{}

\begin{document}
\maketitle

\begin{abstract}
    This article surveys recent results aiming at obtaining refined mapping estimates for the X-ray transform on a Riemannian manifold with boundary, which leverage the condition that the boundary be strictly geodesically convex. These questions are motivated by classical inverse problems questions (e.g. range characterization, stability estimates, mapping properties on Hilbert scales), and more recently by uncertainty quantification and operator learning questions. 
\end{abstract}

\tableofcontents

\newpage

\section{Introduction}

\subsection{Main setting and questions}

Consider $(M,g)$ a Riemannian manifold with boundary $\partial M$ of dimension $n\ge 2$, with unit tangent bundle $SM := \{(\x,v)\in TM,\quad g_\x(v,v)=1\}$ with canonical projection $SM \stackrel{\pi}{\longrightarrow}M$ and boundary $\partial SM = \{(\x,v)\in SM,\ \x\in \partial M\}$. We denote the geodesic flow by $\varphi_t\colon SM\to SM$ defined on the domain 
\begin{align}
    \D := \{ (\x,v) \in SM,\ t\in [-\tau(\x,-v),\tau(\x,v)]  \},
    \label{eq:D}
\end{align}
where $\tau(\x,v)\in [0,\infty]$ is the first time at which $\pi(\varphi_\tau (\x,v))\in \partial M$. Two important assumptions about $M$ are that (i) the boundary is strictly geodesically convex (in the sense of having positive definite second fundamental form) and (ii) $M$ is non-trapping, in the sense that $\sup_{(\x,v)\in SM} \tau(x,v) <\infty$ (i.e., every unit-speed geodesic exits $M$ in finite time). Conditions (i) and (ii) allow one to realize the space of geodesics through $M$ as the manifold of inward-pointing unit vectors $\partial_+ SM$, a manifold with boundary the ``glancing boundary'' $\partial_0 SM$, both defined as 
\begin{align}
    \partial_+ SM &:= \{ (\x,v)\in SM,\ \x\in \partial M,\ \mu(\x,v) \ge 0\}, \label{eq:DplusSM}\\
    \partial_0 SM &:= \{ (\x,v)\in SM,\ \x\in \partial M,\ \mu(\x,v) = 0\}, \label{eq:D0SM}
\end{align}
where $\mu(\x,v):= g_\x(v,\nu_\x)$, and $\nu_\x$ is the unit inner normal at $\x\in \partial M$. %The manifolds $M$, $SM$ and $\partial_+ SM$ are equipped with the Riemannian metric $dVol_g$, the Sasaki volume form $d\Sigma^{2n-1}$ and hypersurface volume form\footnote{here $\delta_{\nu_\x}$ denotes the horizontal lift of $\nu_\x$.} $d\Sigma^{2n-2} = \iota_{\delta_{\nu_\x}} d\Sigma^{2n-1}$, respectively. 
There is a natural footpoint map which sends a point $(\x,v)\in SM$ to the first point in its geodesic past hitting the boundary, namely,
\begin{align}
	F\colon SM \to \partial_+ SM, \qquad F(\x,v) := \varphi_{-\tau(\x,-v)}(\x,v), \qquad (\x,v)\in SM.
	\label{eq:footpoint}
\end{align}
Conditions (i) and (ii) guarantee that $\tau$ and $F$ are smooth on $SM \backslash \partial_0 SM$. 

The main operators under study are the X-ray transform $I_0$ and the backprojection operator $I_0^\sharp$, defined as follows: for $f\in C_c^\infty(M)$ and $g\in C_c^\infty(\partial_+ SM)$, 
\begin{align}
    I_0 f(\x,v) &:= \int_0^{\tau(\x,v)} f(\pi(\varphi_t(\x,v)))\ dt, \qquad (\x,v)\in \partial_+ SM, \label{eq:Xray} \\
    I_0^\sharp g(\x) &:= \int_{S_\x} g(F(\x,v))\ dS(v), \qquad \x\in M. \label{eq:backproj}
\end{align}
Classically (see, e.g., \cite{Sharafudtinov1994}), one can easily see that $I_0$ extends to a bounded operator $L^2(M) \to L^2_\mu (\partial_+ SM)$ with Hilbert space adjoint the bounded extension of $I_0^\sharp \colon L^2_\mu (\partial_+ SM)\to L^2(M)$. Here $M$ is equipped with its Riemannian volume form, and $L^2_\mu (\partial_+ SM)$ refers the $L^2$ space for the volume form $\mu d\Sigma^{2n-2}$, where $d\Sigma^{2n-2}$ is the natural Sasaki area form on $\partial SM$.

% some background
The transform \eqref{eq:Xray} has been thoroughly studied for the past few decades, and is quite well-understood in many respects. Its injectivity holds in various flexible contexts: in two dimensions, {\em simple}\footnote{A Riemannian manifold $(M,g)$ is called {\bf simple} if in addition to (i) and (ii), it satisfies (iii) $(M,g)$ has no conjugate points.} geometries \cite{Mukhometov1975,Paternain2011a}, including further range characterizations and inversion formulas \cite{Krishnan2010,Pestov2004}; in higher dimensions, geometries admitting convex foliations \cite{Stefanov2014,Uhlmanna}. On the smoothing properties of \eqref{eq:Xray} and stability estimates, it is generally understood that, in the absence of conjugate points, the operator $I_0^\sharp I_0$ is an elliptic pseudodifferential operator of order $-1$ on $M^{int}$. One may then derive the following stability estimate, see \cite{Stefanov2004}: if $(\widetilde{M},\tilde{g})$ denotes a simple neighborhood of $M$ with $\tilde{g}|_{M} = g$, and denoting $\widetilde{I_0^\sharp I_0}$ the associated operator on $\widetilde{M}$, we have $(\widetilde{I_0^\sharp I_0} (e_M f))|_M = I_0^\sharp I_0 f$ for any $f\in L^2(M)$, and moreover, 
\begin{align}
	\|f\|_{L^2(M)} \le C \|\widetilde{I_0^\sharp I_0} (e_M f)\|_{H^1(\widetilde{M}^{int})},
	\label{eq:SUstab}
\end{align}
where $e_M\colon L^2(M)\to L^2(\widetilde{M})$ is the extension-by-zero-to-$\widetilde{M}$ operator. 

%%% explain what is undersireable about this: the need for an extension, unclear what can be said globally; how to write this on general spaces? L^2-H^1/2 stability ? Bohr-Nickl

Geometries with lots of symmetries enjoy a refined understanding \cite{Natterer2001,Helgason2010}: in some cases, a full Singular Value Decomposition can be derived \cite{Louis1984,Mishra2019,Mishra2022}, and constructive approaches to tensor tomography \cite{Monard2017a,Monard2015a} can be given. %%%% unfinished 

% motivation for the questions we'll address
Nonetheless, some questions (motivated in the next paragraph) have remained unaddressed: one undesireable feature of \eqref{eq:SUstab} is that it depends on the extension $\widetilde{M}$, and as such does not immediately address an $L^2-H^1$ stability estimate {\bf on} $M$. To address this question, classical Sobolev spaces are in fact not always the right spaces to capture boundary behavior of X-ray transforms on convex surfaces, to formulate stability estimates which are valid up to the boundary of $M$, and to find functional contexts where $I_0^\sharp I_0$ is invertible. The main goal of this article is then to survey recent results addressing these questions, broadly formulated as follows: 

\begin{description}
	\item[Q1:] In what function spaces (Sobolev-type spaces, or Fr\'echet spaces) on $M$ and $\partial_+ SM$ can one sharply formulate forward mapping properties and stability estimates for $I_0$ and $I_0^\sharp$ ?
    \item[Q2:] For some ``natural'' choices of weight functions $w_1$ on $M$ and $w_2$ on $\partial_+ SM$, when can we prove that $I_0^\sharp w_2 I_0 w_1$ is invertible ?
\end{description}

The reason for introducing such weights may become clear in what follows: they have specified (singular) boundary behavior at the boundaries of $M$ and $\partial_{+}SM$, which can make the associated operator $I_0^\sharp w_2 I_0 w_1$ better-behaved (than, say, $I_0^\sharp I_0$), while not drastically changing the properties of the operator $I_0$. 

\subsection{Motivation}

A first reason to consider Q1 and Q2 arises from the classical inverse problems agenda surrounding the X-ray transform \eqref{eq:Xray}, an operator that appears in many applied imaging problems in medicine, material science and geoprospection, see, e.g., the recent topical review \cite{Ilmavirta2019} on the topic. The questions that one may ask are: how to characterize the range of $I_0$? how to invert $I_0$? What functional setting describes the mapping properties of $I_0$? These questions become easy to answer in any context where the operator $I_0^\sharp w_2 I_0 w_1 \colon \E\to\F$ is provably (even better, {\em constructively}) invertible: 
\begin{itemize}
	\item Any $f\in \E$ is recovered from $I_0 w_1 f$ via the formula $f = (I_0^\sharp w_2 I_0 w_1)^{-1} I_0^\sharp w_2 I_0 w_1 f$. In all cases below, $\E$ is large enough to contain $C_c^\infty(M)$ and $w_1$ is smooth, non-vanishing on any compact subset of $M^{int}$. 
	\item To project noisy data onto the range of $I_0$, a question of practical interest, one may use the operator $I_0 w_1 (I_0^\sharp w_2 I_0 w_1)^{-1} I_0^\sharp w_2$, obviously idempotent on $\F$.  
\end{itemize}

Generally, obtaining mapping properties of a given linear operator on {\em Hilbert scales}\footnote{A family of Banach (or Hilbert) spaces $\{(E_\alpha, \|\cdot\|_\alpha)\}_{\alpha_0<\alpha<\beta_0}$ is a {\em Banach scale} in the sense of \cite{Krein1966} if (1) $E_\beta$ is densely embedded in $E_\alpha$ when $\beta>\alpha$ with $\|x\|_\alpha \le_{\alpha,\beta} \|x\|_\beta$ for all $x\in E_\beta$, and (2) interpolation inequalities hold in the sense that there is a function $C(\alpha,\beta,\gamma)$ finite at all points of the domain $\alpha_0\le \alpha<\beta<\gamma\le \beta_0$ such that $\|x\|_\beta\le C(\alpha,\beta,\gamma) \|x\|_\alpha^{\frac{\gamma-\beta}{\gamma-\alpha}} \|x\|_{\gamma}^{\frac{\beta-\alpha}{\gamma-\alpha}}$ for all $x\in E_\gamma$.} that accurately capture that operator's degree of ill-posedness, makes the operator amenable to the most refined theoretical frameworks for addressing practical aspects such as regularization theory \cite{Natterer1984,Engl1996}, or more recently, operator learning \cite{Rastogi2023}.

A more recent and strong motivation for the recent progress on the topic, was a recent theoretical development in the non-parametric, infinite-dimensional Bayesian analysis of inverse problems in the presence of statistical noise, which led to the recent works \cite{Monard2017,Monard2019,Monard2021,Bohr2021a} (some of whose results are quoted below), see also the recent monograph \cite{Nickl2023}. Given a forward operator $A$ (in our case, $A = I_0$), under the assumption that the data $A f$ is polluted by random Gaussian noise, the problem becomes to reconstruct $f$ from $d = A f + \varepsilon$. The idea of a single ``ground truth reconstruction'' loses its meaning, and is replaced by the study of ``all possible reconstructions and their probabilities'' (i.e., in the Bayesian framework, the posterior random variable $f|d$ and its probability density), the design of trustworthy estimators for the 'ground truth' $f$ (in the frequentist viewpoint), and their computability. Many new statistical results in Bayesian nonparametrics have been derived, resting on hypotheses which require a refined understanding of the forward operator $A$, in particular functional contexts where $A^* A$ is invertible, and stability estimates for $A^* A$ or $A$ which accurately capture the order of smoothing of the forward operator. In the present case where $A = I_0$, obtaining these results requires the design of Sobolev type spaces which sharply capture the mapping properties of $I_0$ on the interior, but also at the boundary, and requires a deeper look at specific types of boundary behaviors, and how they map under $I_0$ and $I_0^\sharp$.

%%% unfinished 

%\Fnote{Given a compact self-adjoint operator, can one describe the canonical scale of spaces using Sobolev type spaces?}

\subsection{Normal operators and boundary behavior: some non-trivial aspects to be reconciled} 

To clarify why boundary behavior can be a thorny question in this context, and as a precursor to the following sections, let us cover a few examples:  

\begin{itemize}
    \item Considering the Euclidean disk $\Dm$ in polar coordinates $\rho e^{i\omega}$, a direct calculation shows that $I_0^\sharp I_0 (1) = 2 \int_{0}^{2\pi} \sqrt{1-\rho^2 \cos^2\theta}\ d\theta$. This is the complete elliptic integral of the second kind which, as function on $\Dm$, belongs to $C^\infty + d\log(d) C^\infty$, if $d := 1-\rho^2$, while the function $f(\rho e^{i\omega}) = 1$ belongs to $C^\infty(\Dm)$. 

    \item On the other hand, if $f(\rho e^{i\omega}) = (1-\rho^2)^{-1/2}$, then a direct calculation shows that $I_0^\sharp I_0 f = 1$. This means that some integrands which blow-up near the boundary may be mapped, via $I_0^\sharp I_0$, to functions which are smooth up to the boundary.

    \item More generally, one may prove that the above situation in fact holds in quite general contexts, as the isomorphism result proved in \cite{Monard2017} (discussed in Section \ref{sec:BdM} below) holds for any simple Riemannian manifold with boundary: 
	\begin{align*}
	    I_0^\sharp I_0 \colon d^{-1/2} C^\infty(M) \stackrel{\approx}{\longrightarrow} C^\infty(M).
	\end{align*}
	Moreover in this case, the domain and codomain above can be thought of as intersection spaces of families of Sobolev spaces where one also gets homeomorphic properties 
	\begin{align}
	    I_0^\sharp I_0 \colon H^{s,(-1/2)}(M) \stackrel{\approx}{\longrightarrow} H^{s+1}(M) \qquad s>-1.
	    \label{eq:MNPmapprop}
	\end{align}
	The spaces on the right side are the classical Sobolev spaces and those on the left side are so-called H\"ormander $-1/2$-transmission spaces. Although this result was a first of its kind in this context, the H\"ormander transmission spaces are difficult to work with, and it is unclear how to next characterize $I_0^\sharp I_0 (H^{s+1}(M))$. In addition, it is a classical result that the scale on the right side of \eqref{eq:MNPmapprop} is {\em not} a Hilbert scale \cite{Neubauer1988}, mostly due to boundary issues (elements in these spaces contain ``too many'' traces, preventing the interpolation inequalities). Hence, off-the-shelf results which work best in the context of Hilbert scales (e.g. \cite{Natterer1984,Engl1996,Rastogi2023}) may not be readily available in this context.

    \item One of the ways to alleviate said issues is to introduce a singular weight between $I_0$ and $I_0^\sharp$. Indeed in \cite{Monard2019a}, with $\Dm$ the Euclidean unit disk, the author showed that the operator $I_0^\sharp \frac{1}{\mu} I_0$ enjoyed the following isomorphic mapping properties
	\begin{align*}
	    I_0^\sharp \frac{1}{\mu} I_0 \colon C^\infty(\Dm)\colon C^\infty(\Dm) \stackrel{\approx}{\longrightarrow} C^\infty(\Dm), \qquad I_0^\sharp \frac{1}{\mu} I_0 &\colon \wtH^s(\Dm) \stackrel{\approx}{\longrightarrow} \wtH^s(\Dm), \qquad s\ge 0,
	\end{align*}
	where $\wtH^{s}(\Dm)\subset L^2(\Dm)$ is the domain space of a degenerate elliptic operator $\L$ which is $L^2 (\Dm)$-essentially self-adjoint when equipped with the core domain $C^\infty(\Dm)$. At the cost of introducing the singular weight $\frac{1}{\mu}$, the spaces $\{\wtH^{s}(\Dm)\}_{s\ge 0}$ form a Hilbert scale and allow to capture all mapping properties of $I_0^\sharp \frac{1}{\mu} I_0$ and its iterates.

\end{itemize}

\paragraph{Smoothness on $\partial_+ SM$, domain of $I_0^\sharp$ and range of $I_0$.} As seen above, normal operators (i.e. compositions of $I_0^\sharp$ and $I_0$) have the advantage to map functions on $M$ back to functions on $M$, and becomes invertible in appropriate Sobolev spaces. By contrast, moving to the study of $I_0$ or $I_0^\sharp$ alone requires addressing additional issues, including the infinite-dimensional kernel of $I_0^\sharp$, and proper domains of definition and Sobolev scales on $\partial_+ SM$. 

Understanding the mapping properties of $I_0^\sharp$ alone requires to define spaces on $\partial_+ SM$ which depend on the {\em scattering relation} $\alpha\colon \partial SM \to \partial SM$ of the associated manifold, defined by 
\begin{align}
	\alpha(x,v) = \varphi_{\tau(x,v)} (x,v), \qquad (x,v)\in \partial_+ SM. 
	\label{eq:scatrel}
\end{align}
Then for $u\in L^2(\partial_+ SM, \mu\ d\Sigma^{2n-2})$, one may define $A_\pm u\in L^2(\partial SM, |\mu|\ d\Sigma^{2n-2})$ by 
\begin{align}
	u(x,v) = \left\{ 
	\begin{array}{ll}
		u(x,v), & (x,v)\in \partial_+ SM \\
		u(\alpha(x,v)), & (x,v)\in \partial_- SM.
	\end{array}
	\right.
	\label{eq:Apm}
\end{align}
At the smooth level, as 'natural' domains for $A_\pm$ to land on $C^\infty(SM)$, it is customary to define 
\begin{align}
	C_{\alpha,\pm}^\infty(\partial_+ SM) := \{u\in C^\infty(\partial_+ SM), \qquad A_\pm u \in C^\infty(\partial SM)\}.
	\label{eq:Calphapm}
\end{align}

The space $C_{\alpha,+}^\infty$ was first introduced (and denoted $C_\alpha^\infty$ there) during the resolution of the boundary rigidity conjecture for simple surfaces by Pestov and Uhlmann in \cite{Pestov2005}. In particular, via the characterization
\begin{align}
	C_{\alpha,+}^\infty(\partial_+ SM) := \{u\in C^\infty(\partial_+ SM),\quad u\circ F \in C^\infty(SM)\}
	\label{eq:PUcharac}
\end{align}
proved there using fold theory, and since $I_0^\sharp u = \pi_\star (u\circ F)$, this immediately implies that 
\begin{align}
	I_0^\sharp (C_{\alpha,+}^\infty(\partial_+ SM))\subset C^\infty(M).
	\label{eq:PUI0sharp}
\end{align}
For the purpose of the current review, the take-home message is that 'natural' spaces for $I_0^\sharp$ (and, later in Sec. \ref{sec:phgonmwb}, 'good' coordinate systems for capturing the mapping properties of $I_0$) need to depend on the scattering relation of the underlying manifold. As a remark which is fundamental in the resolution of the boundary rigidity of surfaces, the same authors show, via microlocal arguments, that if $(M,g)$ is a simple surface, \eqref{eq:PUI0sharp} is in fact an equality. %\Fnote{When can one find a preimage in $I_0 (d^{1/2}C^\infty)$?}

Moving on to the description of the range of $I_0$, the scattering relation \eqref{eq:scatrel} plays again a fundamental role on simple surfaces, as is shown by the same authors in \cite{Pestov2004}. To state their result, we let $A_\pm^*$ be the adjoint of \eqref{eq:Apm} with expression 
\begin{align*}
	A_\pm u (x,v) = u(x,v) \pm u(\alpha(x,v)), \qquad (x,v)\in \partial_+ SM.
\end{align*}
On the fibers of $SM$ (and in particular, of $\partial SM$), we also define the {\em fiberwise Hilbert transform}\footnote{See \cite{Pestov2004} for a definition. It is occulted here as it will not be used.} $H\colon L^2(SM)\to L^2(SM)$, splitting as $H = H_+ + H_-$, where $H_{+/-}$ is the restriction of $H$ to fiberwise even/odd functions. Then \cite[Theorem 4.4]{Pestov2004} states that on any simple surface $(M,g)$, 
\begin{align}
	I_0 (C^\infty(M)) = P_- (C_{\alpha,+}^\infty(\partial_+ SM)), \qquad \text{where}\quad P_\pm := A_-^* H_\pm A_+.
	\label{eq:PUrange}
\end{align}
The author proved in \cite[Theorem 2.3]{Monard2015a} that such a characterization, in the Euclidean case, was equivalent to the classical Helgason-Ludwig-Gelfand-Graev moment conditions. On general simple surfaces, the operator $P_-$ remains somewhat mysterious, though in cases with symmetries (namely, simple geodesic disks in constant curvature), a full Singular Value Decomposition (SVD) can be derived, see \cite[Sec. 3]{Mishra2019}, allowing in turn a sharp understanding of the infinite-dimensional co-kernel of $I_0$. Further, the operator $P_-$ somewhat presents the disadvantage to not produce a way to project noisy data onto the range of $I_0$. In fact, for symmetry reasons, $P_- \circ P_- = 0$. To address this, the author introduced in \cite{Monard2015a} the operator $C_- := \frac{1}{2} A_-^* H_- A_- \colon C_{\alpha,-,-}^\infty(\partial_+ SM) \to C_{\alpha,-,-}^\infty(\partial_+ SM)$, where for $\sigma = \pm$, we define
\begin{align}
	C_{\alpha,\sigma,\pm}^\infty (\partial_+ SM) = \{u \in C_{\alpha,\sigma}^\infty(\partial_+ SM),\ u(\x,v) = \pm u(\alpha(\x,-v))\}.
	\label{eq:Cpmpm}
\end{align}
On simple geodesic disks in constant curvature, the SVD of $C_-$ can also be computed, see \cite{Monard2015a,Mishra2019}, and the range characterization reads, in terms of $C_-$,
\begin{align*}
	I_0 (C^\infty(M)) = \ker C_-.
\end{align*}
Moreover on such surfaces, the operator $Id + C_-^2$ is an $L^2(\partial_+ SM, d\Sigma^2)$-projector onto the range of $I_0$. Sobolev versions of the range descriptions above will also be given in the next sections.

\subsection{Outline} 

The study of Q1 and Q2 involves several results which make use of various toolboxes: Fourier-analytic, microlocal, geometric microlocal, energy-identity based. In what follows, we will first cover examples in Section \ref{sec:EuclDisk} of symmetric manifolds (Euclidean disk and constant curvature disks). In such cases, one knows Singular Value Decompositions of various weighted realizations of $I_0$, as well as relations with distinguished differential operators, both of which help inform the design of functional settings that sharply answer Q1 and Q2. We then move to results on forward mapping properties (and stability estimates) which hold for larger classes of manifolds, covering results for $I_0$ in Section \ref{sec:forward_Xray}, and results for $I_0^\sharp$ in Section \ref{sec:forward_backproj}. Combining these results, we will first formulate some consequences and a conjecture in Section \ref{sec:isomorphism}, regarding contexts where operators of the form $I_0^\sharp w_2 I_0 w_1$ can be made isomorphisms of $C^\infty(M)$, and we will discuss the cases where the conjecture is true, and the methodology to prove them. Finally, in Section \ref{sec:atrt}, we will discuss how such results can be extended to the context of {\em attenuated X-ray transforms}, along with their range of applicability. We conclude with some perspectives in Section \ref{sec:perspectives}.

\paragraph{Preliminary comments.} Results are sometimes formulated as theorems, sometimes incorporated in the text for the sake of the narrative. Proofs will be given at a sketchy level, and an interested reader is often encouraged to look at the original article for details. Some results will be at times referencing results appearing in the recent book \cite{Paternain2023}, an important pedagogical resource in the author's view. When this is the case, we will showcase the main results and redirect the interested reader there to minimize overlap.

%2D, unattenuated and attenuated X-ray transforms.
%We will restrict to the case of incomplete surfaces with convex boundary, putting the emphasis on what this brings us. 

%%%%%%%%%%%%%%%%%%%%%%%%%%%%%%%%%%%%%%%%%%%%%%%%%%%%%%%%%%%%%%%%%% SYMMETRIC CASES
\section{Euclidean disk and symmetric cases as reference models} \label{sec:EuclDisk}

In order to construct sharp functional settings where to capture mapping properties, we first draw intuition from realizations of the X-ray transform where the Singular Value Decomposition (SVD) is explicitly known. This includes the case of the X-ray transform on the Euclidean disk (Sec. \ref{sec:gammazero}) and its singularly weighted versions (Sec. \ref{sec:gammanonzero}), as well as analogous situations on any simple geodesic disk in constant curvature (Sec. \ref{sec:CCDs}). 

While the SVD in itself helps describe the range in distinguished bases, the work in what follows also consists in identifying when the spaces that result can be defined in terms of operators which may be amenable to be generalized to non-symmetric cases.

\subsection{Euclidean disk - unweighted $L^2$} \label{sec:gammazero}

%So far, most results do not have to worry about the fact that $I_0$ has infinite-dimensional co-kernel, because the operators most often looked at are normal realizations. We now give more precisions on what is known on $I_0$.

%In several contexts, the Singular Value Decompositions of $I_0\colon L^2(M) \to L^2 (\partial_+ SM)$ and its weighted versions are known explicitly. For any such context where we have orthonormal systems $(u_{m}, v_{m}, \sigma_{m})_{m\ge 0}$ such that $I_0 u_{m} = \sigma_{m} v_{m}$ for all $m$, one obtains the obvious range characterization 
%\begin{align}
%    I_0 (L^2(M)) = \left\{ u = \sum_{m\ge 0} \sigma_m a_m v_m, \qquad \sum_{m\ge 0} |a_m|^2 <\infty \right\}. 
%    \label{eq:naiverange}
%\end{align}
%Although sharp, for purposes of, e.g. generalizability, it is desireable to be able to (i) describe the right-hand-side in terms of other operators, and (ii) write similar range characterizations when the integrand has better regularity than $L^2$. In what follows, we will present operators and associated Sobolev scales which fulfill these goals. 

Let $\Dm = \{z = \rho e^{i\omega} \in \Cm,\ |z|^2 \le 1 \}$ be the unit disk equipped with the standard Euclidean metric $|dz|^2$. We parameterize the inward boundary $\partial_+ S\Dm$ (defined in \eqref{eq:DplusSM}) as $\Sm^1_\beta \times [-\pi/2,\pi/2]_\alpha$, where $\beta$ parameterizes the boundary point $e^{i\beta}$ and $\alpha$ parameterizes the inward pointing vector $e^{i(\beta+\alpha+\pi)}$ above the basepoint $e^{i\beta}$. In these coordinates, $d\Sigma^2 = d\beta d\alpha$. The unique oriented geodesic passing through $(\beta,\alpha)\in \partial_+ S\Dm$ has equation 
\begin{align*}
	\gamma_{\beta,\alpha}(t) = e^{i\beta} + t e^{i(\beta+\pi+\alpha)}, \quad t\in [0,2\cos \alpha]. 
\end{align*}
Each unoriented geodesic appears twice, in particular we have the symmetry
\begin{align}
	\gamma_{\beta+\pi+2\alpha,-\alpha}(t) = \gamma_{\beta,\alpha} (2\cos\alpha-t), \qquad t\in [0,2\cos\alpha].
	\label{eq:symmetry}
\end{align}
The X-ray transform $I_0 f (\beta,\alpha) = \int_0^{2\cos\alpha} f(\gamma_{\beta,\alpha}(t))\ dt$ makes sense for $f\in C_c^\infty(\Dm^{int})$ and extends into a bounded operator 
\begin{align}
	I_0 \colon L^2(\Dm, |dz|^2) \to L_+^2(\partial_+ S\Dm, d\beta d\alpha),	
	\label{eq:I0L2}
\end{align}
with Hilbert space adjoint $I_0^\sharp \frac{1}{\mu}$, where $\mu = \cos\alpha$. The '$+$' subscript in the co-domain of $I_0$ indicates evenness with respect to the map $(\beta,\alpha)\mapsto (\beta+\pi+2\alpha, -\alpha)$ which, in light of \eqref{eq:symmetry}, arises because $I_0 f$ is even with respect to geodesic orientation.

\subsubsection{Singular Value Decomposition} We first briefly describe the SVD of $I_0$ on the Euclidean disk. 
Define, on $\partial_+ S\Dm$, 
\begin{align}
    \psi_{n,k} := \frac{(-1)^n}{4\pi} e^{i(n-2k)(\beta+\alpha)} \left( e^{i(n+1)\alpha} + (-1)^n e^{-i(n+1)\alpha} \right), \qquad n\ge 0, \qquad k\in \Zm.
    \label{eq:psink}
\end{align}
These functions belong to $C_{\alpha,-,+}^\infty( \partial_+ S\Dm)$ and make up a Hilbert basis of $L^2_+ (\partial_+ S\Dm)$ (the co-domain in \eqref{eq:I0L2}), each with norm $\|\psi_{n,k}\|_{L^2(\partial_+ S\Dm)}^2 = 1/4$. One may then show that $I_0^\sharp \left[\frac{\psi_{n,k}}{\mu}\right] = 0$ if $k\notin \{0, \dots, n\}$ and otherwise we define $Z_{n,k} := I_0^\sharp \left[\frac{\psi_{n,k}}{\mu}\right]$ for $n\ge 0$ and $0\le k\le n$. The latter functions coincide with the Zernike polynomials as defined in \cite{Kazantsev2004}, and make up an orthogonal Hilbert basis of $L^2(\Dm)$, each with norm $\|Z_{n,k}\|^2_{L^2(\Dm)} = \frac{\pi}{n+1}$. The SVD of $I_0$ is then immediately given by (see, e.g., \cite{Louis1984,Mishra2019})
\begin{align}
	\widehat{Z_{n,k}},\ \widehat{\psi_{n,k}},\ \frac{\sqrt{4\pi}}{\sqrt{n+1}}, \qquad n\ge 0,\ 0\le k\le n,
	\label{eq:SVDI0}
\end{align}
and we can na\"ively formulate the range as follows
\begin{align}
    I_0 (L^2 (\Dm)) = \left\{ \sum_{n\ge 0} \sum_{k=0}^n a_{n,k} \frac{\sqrt{4\pi}}{\sqrt{n+1}} \widehat{\psi_{n,k}}, \qquad \sum_{n\ge 0}\sum_{k=0}^n |a_{n,k}|^2 <\infty \right\} \subset L^2_+ (\partial_+ S\Dm).
    \label{eq:naiverange2}
\end{align}

To describe the last right-hand side without the use of special functions, there are two things to address:
\begin{itemize}
	\item[(i)] To find Sobolev scales that capture the spectral decay incurred by the '$\frac{1}{2}$-smoothing'.
	\item[(ii)] To find operators that capture the infinite-dimensional co-kernel spanned by those $\psi_{n,k}$ for which $k\notin \{0, \dots, n\}$. 
\end{itemize}

% construct Sobolev scales
\subsubsection{Some non-standard Sobolev scales} 

To address (i), we now introduce the operators $T = \partial_\beta-\partial_\alpha$ on $\partial S\Dm$, and on $\Dm$, 
\begin{align}
	\L := - \rho^{-1} \partial_\rho \left( (1-\rho^2) \rho\partial_\rho \right) - \rho^{-2} \partial_\omega^2 + id.
	\label{eq:Lzero}
\end{align}
One may show that the operator $-T^2$, equipped with the domain $C_{\alpha,-,+}^\infty(\partial_+ S\Dm)$ (dense in $L^2_+ (\partial_+ S\Dm)$) is essentially $L_+^2(\partial_+ S\Dm)$ self-adjoint (see \cite[Theorem 7, case $\gamma=0$]{Mishra2022}), and that $\L$, equipped with the domain $C^\infty(\Dm)$, is essentially $L^2(\Dm)$-self-adjoint (see \cite[Theorem 6, case $\gamma=0$]{Mishra2022}). Moreover, the important intertwining property holds: $I_0^\sharp \frac{1}{\mu} \circ (-T^2) = \L \circ I_0^\sharp \frac{1}{\mu}$ (see \cite[Theorem 9]{Monard2019a}). Together with the fact that $(-T^2) \psi_{n,k} = (n+1)^2 \psi_{n,k}$ (a direct calculation), we can obtain a full eigenvalue decomposition, as 
\begin{align*}
    \L Z_{n,k} = \L I_0^\sharp \frac{\psi_{n,k}}{\mu} = I_0^\sharp \frac{1}{\mu} (-T^2)\psi_{n,k} = I_0^\sharp \frac{1}{\mu} (n+1)^2 \psi_{n,k} = (n+1)^2 Z_{n,k},
\end{align*}
for any $n\ge 0$ and $k\in \{0, \dots, n\}$. It may then be natural to construct Sobolev spaces in terms of the operator $\L$. Namely, we define the Zernike-Sobolev scale
\begin{align}
	\wtH^s(\Dm) = \text{dom}(\L^{s/2}) = \left\{ \sum_{n\ge 0} \sum_{k=0}^n f_{n,k} \widehat{Z_{n,k}}, \quad \sum_{n\ge 0} \sum_{k=0}^n (n+1)^{2s} |f_{n,k}|^2 \right\}.
	\label{eq:SobolevZernike}
\end{align}
On the data side one may also define
\begin{align}
    H_{T,+}^{s} (\partial_+ S\Dm) := \text{dom} ( (-T^2)^{s/2}) = \left\{ \sum_{n\ge 0} \sum_{k\in \Zm} a_{n,k} \widehat{\psi_{n,k}}, \qquad \sum_{n\ge 0}\sum_{k=0}^n (n+1)^{2s} |a_{n,k}|^2 <\infty \right\}
    \label{eq:scaledata}
\end{align}
and notice that the right side of \eqref{eq:naiverange2} is a closed subspace of $H_{T,+}^{1/2}(\partial_+ S\Dm)$, with orthocomplement the $H_{T,+}^{1/2}$ span of $\{\psi_{n,k},\ n\ge 0,\ k\notin \{0,\dots, n\}\}$.

\paragraph{Operators that help capture the co-kernel of $I_0$.} To address (ii), the operator $C_- := \frac{1}{2} A_-^* H A-$ defined at the end of the introduction has spectrum $\pm i$ on $\{\psi_{n,k},\ n\ge 0,\ k\notin \{0,\dots, n\}\}$ and vanishes on the range of $I_0$. Moreover, $[C_-, -T^2] = 0$, so that $C_-\colon H_{T,+}^{s+\frac{1}{2}} (\partial_+ S\Dm) \to H_{T,+}^{s+\frac{1}{2}} (\partial_+ S\Dm)$ is bounded for any $s\in \Rm$, and its kernel is exactly the $H_{T,+}^{s+\frac{1}{2}}$-span of $\{\psi_{n,k},\ n\ge 0,\ 0\le k\le n\}$. In the end we arrive at the characterization, see \cite[Theorem 6]{Monard2019a}: 
\begin{align}
    I_0 (\wtH^s(\Dm)) = H_{T,+}^{s+\frac{1}{2}} (\partial_+ S\Dm) \cap \ker (C_-) = H_{T,+}^{s+\frac{1}{2}} (\partial_+ S\Dm) \cap \text{ran} (P_-),
    \label{eq:rangecharac}
\end{align}
moreover there is a constant $C_s$ such that the {\em isometric} estimate holds: 
\begin{align}
	\|f\|_{\wtH^s(\Dm)} = C_s \|I_0 f\|_{H_{T,+}^{s+\frac{1}{2}} (\partial_+ S\Dm)} \qquad \text{for all } f\in \wtH^s(\Dm).
	\label{eq:I0isom}
\end{align}
The isometric feature of \eqref{eq:I0isom} is what makes the {\em anisotropic} Sobolev scale $H_{T,+}^{s+\frac{1}{2}} (\partial_+ S\Dm)$ somewhat desirable. However it should be noted that one may also define a more 'isotropic' Dirichlet Sobolev scale on $\partial_+ S\Dm$: there, the elliptic operator $-T^2 - \partial_\beta^2$ is essentially $L^2_+ (\partial_+ S\Dm)$-self-adjoint on $C_{\alpha,-,+}^\infty(\partial_+ S\Dm)$, with closure which we'll denote $(-T^2 - \partial_\beta^2)_D$ (where the subscript $D$ indicates that it can be thought of as a Dirichlet realization) and having $\{\psi_{n,k}\}_{n,k}$ as a complete set of eigenfunctions. Then one may define the scale of spaces
\begin{align}
	H^s(\partial_+ S\Dm) := \text{dom} ((-T^2 - \partial_\beta^2)_D^{s/2}), \qquad s\in \Rm.
	\label{eq:scale_iso}
\end{align}
Then \cite[Proposition 16]{Monard2019a} shows that on $I_0 (\wtH^s(\Dm))$, the topologies $H_{T,+}^{s+\frac{1}{2}} (\partial_+ S\Dm)$ and $H^{s+\frac{1}{2}} (\partial_+ S\Dm)$ are equivalent and hence one may also show that for any $s\in \Rm$, there exist positive constants $C'_{s}, C''_{s}$ such that 
\begin{align*}
	C'_{s} \|I_0 f\|_{H^{s+\frac{1}{2}} (\partial_+ S\Dm)} \le \|f\|_{\wtH^s(\Dm)} \le C''_{s} \|I_0 f\|_{H^{s+\frac{1}{2}} (\partial_+ S\Dm)} \qquad \text{for all } f\in \wtH^s(\Dm).
\end{align*}

\subsubsection{Isomorphism property} 

Now considering the $L^2(\Dm)$-self-adjoint operators $I_0^\sharp \frac{1}{\mu} I_0$, SVD \eqref{eq:SVDI0} immediately implies the isometric estimate 
\begin{align*}
	I_0^\sharp \frac{1}{\mu} I_0\colon \wtH^{s} (\Dm) \stackrel{\approx}{\longrightarrow} \wtH^{s+1} (\Dm), \qquad s\ge 0.
\end{align*}
Upon proving in \cite[Theorem 12]{Monard2019a} that $\cap_{s\in \Rm} \wtH^s(\Dm) = C^\infty(\Dm)$, then $I_0^\sharp \frac{1}{\mu} I_0$ is an automorphism of $C^\infty(\Dm)$, a special case of Theorem \ref{thm:Cinf} below. More precisely, $C^\infty(\Dm)$ can be equipped with the graded family of semi-norms $\{\wtH^s(\Dm)\}_{s\in \Nm_0}$, on which $I_0^\sharp \frac{1}{\mu} I_0$ is a {\em tame} linear map with tame inverse, see also Remark \ref{rem:tame} below. 

\subsubsection{Functional relations} 

From looking at how $\L$ and $I_0^* \frac{1}{\mu} I_0$ act on the Zernike basis, we can deduce the relation 
\begin{align}
    (I_0^* \frac{1}{\mu} I_0)^2 \L_0 = (4\pi)^2 id \quad \text{on } C^\infty(\Dm). 
    \label{eq:funcRel}
\end{align}
Relation \eqref{eq:funcRel} can be viewed as the ``surface-with-boundary'' analogue of the famous formula
\begin{align*}
	(R^t R)^2 (-\Delta) = (4\pi)^2 id	
\end{align*}
on $\Rm^2$, where $Rf(\theta,s) = \int_\Rm f(s\hat\theta + t\hat\theta^\perp)$, $(\theta,s)\in \Sm^1\times\Rm$, denotes the classical Radon transform and $R^t g(x) = \int_{\Sm^1} g(\theta,x\cdot\hat\theta)\ d\theta$ is its formal $L^2(\Rm^2)-L^2(\Rm\times \Sm^1)$-adjoint.

\subsection{Euclidean disk - singularly weighted $L^2$} \label{sec:gammanonzero}

A continuum of weighted realizations of $I_0$ have had their Singular Value Decomposition known for a long time, and this allows to give other functional settings where the X-ray transform or its normal realizations can be understood. We summarize here the results of \cite{Mishra2022}. 

Let $d = 1-\rho^2$ be a distinguished boundary defining function for $\Dm$. For $\gamma>-1$, one may define the weighted transform $I_0 d^\gamma$ via the formula $I_0 d^\gamma (f) := I_0 (d^\gamma f)$ for $f\in C^\infty(\Dm)$. Such a transform is a bounded operator in the following context 
\begin{align}
	I_0 d^\gamma \colon L^2(\Dm, d^\gamma\ |dz|^2) \to L^2(\partial_+ S\Dm, \mu^{-2\gamma}\ d\Sigma^2),
	\label{eq:dgammaXray}
\end{align}
with Hilbert space adjoint $I_0^\sharp \mu^{-2\gamma-1}$. The case $\gamma=0$ recovers Section \ref{sec:gammazero} which, as we will see below, was special in some respects. We briefly present the SVD of \eqref{eq:dgammaXray} as presented in \cite[Theorem 1]{Mishra2022}, though it dates back to at least \cite{Louis1984} and, in general dimension, applies to the {\em Radon} transform on singularly weighted $L^2$ spaces.

\subsubsection{Singular Value Decomposition} 

Similarly to \eqref{eq:psink}, define on $\partial_+ S\Dm$ 
\begin{align}
	\psi_{n,k}^\gamma := \mu^{2\gamma+1} \frac{(-1)^n}{2\pi} e^{i(n-2k)(\beta+\alpha)} L_n^\gamma (\sin\alpha), \qquad n\ge 0, \qquad k\in \Zm
	\label{eq:psinkgamma}
\end{align}
where $L_n^\gamma\colon [-1,1]\to \Rm$ is the $n$-th orthogonal polynomial for the weight $(1-x^2)^{\gamma+\frac{1}{2}}$, fixing any normalization condition. Then, following a similar structure to the case $\gamma=0$, one finds that $I_0^\sharp \mu^{-2\gamma-1}\colon L^2(\partial_+ S\Dm, \mu^{-2\gamma})\to L^2(\Dm, d^\gamma)$ has kernel the $L_\gamma^2$-span of $\{\widehat{\psi_{n,k}^\gamma}, n\ge 0, \quad k\notin \{0,\dots,n\}\}$. In particular, the operator \eqref{eq:dgammaXray} has infinite-dimensional co-kernel. Moreover, upon defining
\begin{align}
	Z_{n,k}^\gamma := I_0^\sharp \mu^{-2\gamma-1} \psi_{n,k}^\gamma, \qquad n\ge 0, \quad 0\le k\le n,
	\label{eq:gZernike}
\end{align}
the SVD of $I_0 d^\gamma$ co-restricted to $L^2(\partial_+ S\Dm, \mu^{-2\gamma}\ d\Sigma^2) / \ker I_0^\sharp \mu^{-2\gamma-1}$ is given by $(\widehat{Z_{n,k}^\gamma}, \widehat{\psi_{n,k}^\gamma}, \sigma_{n,k})$, where 
\begin{align}
	\sigma_{n,k} := \frac{2^{\gamma+1}\sqrt{\pi}}{(n+1)^{1/2}} \left(\frac{B(n-k+1+\gamma,k+1+\gamma)}{B(n-k+1,k+1)}\right)^{1/2}, \qquad n\ge 0, \quad 0\le k\le n,
	\label{eq:singularValues}
\end{align}
and where $B(x,y) = \int_0^1 u^{x-1}(1-u)^{y-1}\ du$ denotes the Beta function. The functions defined in \eqref{eq:gZernike} are, up to scaling, the generalized Zernike disk polynomials as defined in \cite{Wuensche2005}, and they enjoy being eigenfunctions of a degenerate elliptic operator $\L_\gamma$ given in polar coordinates $z = \rho e^{i\omega}$ by
\begin{align}
	\L_\gamma := - \rho^{-1} (1-\rho^2)^{-\gamma} \partial_\rho \left( \rho (1-\rho^2)^{\gamma+1} \partial_\rho \right) - \rho^{-2} \partial_\omega^2 + (1+\gamma)^2 id.
	\label{eq:Lgamma}
\end{align}
Note that $\L_0$ is the operator defined earlier in \eqref{eq:Lzero}. In a similar fashion as above, one may show that $\L_\gamma Z_{n,k}^\gamma = (n+1+\gamma)^2 Z_{n,k}^\gamma$, and this motivates the definition of the following Sobolev spaces 
\begin{align}
    \wtH^{s,\gamma} (\Dm) = \text{dom} (\L_\gamma^{s/2}) = \left\{ u = \sum_{n\ge 0} \sum_{k=0}^n a_{n,k} \widehat{Z_{n,k}^\gamma}, \qquad \sum_{n,k} (n+1+\gamma)^{2s} |a_{n,k}|^2 <\infty\right\}.
    \label{eq:Hsgamma}
\end{align}
Combining this definition and the SVD of $I_0 d^\gamma$, it immediately follows that 
\begin{align}
    I_0 d^\gamma \left( \wtH^{s,\gamma}(\Dm) \right) = \left\{ u = \sum_{n\ge 0} \sum_{k=0}^n a_{n,k} \sigma_{n,k}^\gamma \widehat{\psi_{n,k}^\gamma}, \qquad \sum_{n,k} (n+1+\gamma)^{2s} |a_{n,k}|^2 <\infty\right\}. 
    \label{eq:rangegammanotzero}
\end{align}
The issues (i) and (ii) addressed in Section \ref{sec:gammazero} can then be considered here as well, however if $\gamma\ne 0$, neither is straighforward: regarding (i) if $\gamma\ne 0$, the dependence of the singular values $\sigma_{n,k}^\gamma$ on $k$ makes it difficult to express the spectral decay on $\partial_+ S\Dm$ on a Sobolev scale defined in terms of powers of differential operators, see \cite[Remark 5]{Mishra2022}; regarding (ii), whether there exist generalizations of operators $P_-$ or $C_-$ with a simple expression, allowing to describe the co-kernel of $I_0 d^\gamma$, is an open question at present. 

\subsubsection{Isomorphism property}

%On a route that is methodologically orthogonal to the previous section, one can provide other isomorphic estimates on simple surfaces with lots of symmetries: simple geodesic disks of constant curvature. This is done in \cite{Monard2019a,Mishra2022}. 

%We start with the Euclidean (closed) unit disk $\Dm = \{z\in \Cm,\ |z|^2 \le 1\}$, with boundary defining function $d := 1-|z|^2$. We describe its inward pointing boundary $\partial_+ S\Dm$ in fan-beam coordinates $\Sm^1_\beta \times [-\pi/2,\pi/2]_\alpha$, where $\beta$ parameterizes a boundary point $z(\beta) = e^{i\beta}$, and $\alpha$ parameterizes the inward-point vector $e^{i(\beta+\pi+\alpha)}$. Then an $\alpha$-boundary defining function for $\partial_+ S\Dm$ is $\mu = \cos\alpha$.  

Generalizing what was presented in Section \ref{sec:gammazero}, it is proved in \cite{Mishra2022} the following: 
\begin{theorem}[Theorem 2 in \cite{Mishra2022}]\label{thm:Cinf}
    Let $\gamma>-1$ be fixed. Then the operator $I_0^\sharp \mu^{-2\gamma-1} I_0 d^\gamma$ is an isomorphism of $C^\infty(\Dm)$.     
\end{theorem}

\begin{proof}[Proof of Theorem \ref{thm:Cinf} (roadmap)] Similarly to the case $\gamma=0$, the statement proved is somewhat more precise: upon proving that $\cap_{s} \wtH^{s,\gamma}(\Dm) = C^\infty(\Dm)$, one may equip $C^\infty(\Dm)$ with the graded family of semi-norms $\{\|\cdot\|_{\wtH^{s,\gamma}},\ s\ge 0\}$, and prove that $I_0^\sharp \mu^{-2\gamma-1} I_0 d^\gamma$ is a tame linear operator with tame inverse. 
	%On the Sobolev side, a closer inspection to the proof shows that there exists a second-order operator $\L_\gamma$, elliptic at interior points but whose ellipticity degenerates at the boundary, to first order along the normal direction. It can be shown that $(\L_\gamma, C^\infty(\Dm))$ is essentially self-adjoint with respect to the Hilbert space $L^2(\Dm, d^\gamma\ dV_E)$ (see, e.g. \cite[Theorem 6]{Mishra2022}), and its closure (which we'll denote again by $\L_\gamma$ here) has a known spectral decomposition (in fact, discrete spectrum, with eigenfunctions the generalized Zernike disk polynomials). Then for any $s\ge 0$, one may define 
	%\begin{align}
	%	\wtH^{s,\gamma}(\Dm) = \text{dom}(\L_\gamma^{s/2})
	%	\label{eq:Hgammas}
	%\end{align}
	Proving that $\cap_{s} \wtH^{s,\gamma}(\Dm) = C^\infty(\Dm)$ follows from the following continuous inclusions, see \cite[Lemma 16]{Mishra2022}:
	\begin{itemize}
		\item We have $\wtH^{s,\gamma}(\Dm) \hookrightarrow C(\Dm)$ whenever $s>\max (\gamma,\frac{3}{2}\gamma)$. 
		\item For $k\ge 1$, we have $\wtH^{s,\gamma}(\Dm) \hookrightarrow C^k(\Dm)$ whenever $s>2+ \frac{3}{2}\gamma + \frac{5}{2} k$.
	\end{itemize}	
	Then the tameness of $I_0^\sharp \mu^{-2\gamma-1} I_0 d^\gamma$ and its inverse follow directly from the asymptotics of the singular values \eqref{eq:singularValues}, of the form
	\begin{align*}
		C_1 (n+1)^{\min(-1,-1-\gamma)} \le (\sigma_{n,k}^2) \le C_2 (n+1)^{\max(-1,-1-\gamma)}, \qquad n\ge 0, \qquad 0\le k\le n,
	\end{align*}
	for some constants $C_1,C_2$ independent of $n,k$, see \cite[Lemma 12]{Mishra2022}. This implies the following mapping properties: 
	\begin{align}
		\wtH^{s+\max (1,1+\gamma), \gamma}(\Dm) \subset I_0^\sharp \mu^{-2\gamma-1} I_0 d^\gamma(\wtH^{s,\gamma}(\Dm))\subset \wtH^{s+\min (1,1+\gamma), \gamma}(\Dm), \qquad s\ge 0.
		\label{eq:SobolevMapping}
	\end{align}
\end{proof}

All auxiliary results (e.g., SVD, continuous embeddings), make extensive use of knowledge about the generalized disk Zernike polynomials \cite{Wuensche2005}. What is especially needed is their $L^2_\gamma$ norms, (approximate) uniform estimates, how they transform under differentiation, and under $I_0 d^\gamma$. There, representation-theoretic properties inform us on what operators act in a sparse way on these bases. 

\begin{remark} A more in-depth study of the family of degenerate-elliptic operators $\{\L_\gamma,\ \gamma\in \Rm\}$ is undertaken in \cite{Monard2023}, putting the spaces $\wtH^{s,\gamma}(\Dm)$ in a broader perspective, and providing a framework to understand other self-adjoint realizations of $\L_\gamma$ whenever they exist. 	
\end{remark}

\begin{remark}\label{rem:tame}
	In the language of tame Fr\'echet spaces and tame maps: for any $\gamma>-1$, $C^\infty(\Dm)$ may be equipped with the graded countable family of seminorms $\{\|\cdot\|_{\wtH^{s,\gamma}},\ s\in \Nm\}$, giving it a structure of a tame Fr\'echet space, and the operator $I_0^\sharp \mu^{-2\gamma-1} I_0 d^\gamma$ is tame with tame inverse for that specific scale. For inverse problems, this is equivalent to saying that the operator has fixed degree of continuity and ill-posedness. 

	This result in particular shows that, for example when $\gamma=0$, $I_0^\sharp \frac{1}{\mu} I_0$ cannot be made tame-with-tame-inverse when $C^\infty(\Dm)$ is equipped with the classical Sobolev semi-norms $H^k (\Dm)$, so that classical spaces are not adapted to capturing mapping properties of the X-ray transform. Indeed, if this were the case, then the graded Fr\'echet spaces $(C^\infty(\Dm), \{\|\cdot\|_{\wtH^k}\}_{k\in \Nm_0})$ and $(C^\infty(\Dm), \{\|\cdot\|_{H^k}\}_{k\in \Nm_0})$ would be {\em tamely equivalent} in the sense that there would exist fixed integers $\ell, \ell'\in \Zm$ such that $\|\cdot\|_{\wtH^k} \lesssim_k \|\cdot\|_{H^{k+\ell}}$ and $\|\cdot\|_{H^k} \lesssim_k \|\cdot\|_{\wtH^{k+\ell'}}$, independently of $k$. 

	However, this last fact is not true: on the unit disk with boundary defining function $d = 1-\rho^2$, consider the function $f = d^{2p+1/2}$ for any $p\in \Nm_0$. By repeated differentiation, it is easy to see that $f\in \wtH^{2p+2p}(\Dm)$ but $f\notin H^{2p+2}(\Dm)$. This invalidates the possibility of inequalities $\|\cdot\|_{H^{k}}\lesssim \|\cdot\|_{\wtH^{k+\ell}}$ for some $\ell$ independent of $k$.
\end{remark}

\paragraph{Functional relations.} In the case where $\gamma\ne 0$, a more general relation than \eqref{eq:funcRel} exists between the {\em three} operators $I_0^\sharp \mu^{-2\gamma-1} I_0 d^\gamma$, $\L_\gamma$ and $\partial_\omega^2$, a direct consequence of \eqref{eq:singularValues} and the fact that the spectral parameters $n$ and $k$ can be respectively quantized into ${\cal D}_\gamma := \L_\gamma^{1/2}-\gamma-1$ and $({\cal D}_\gamma-\frac{1}{i}D_\omega)/2$, specifically giving
\begin{align}
	I_0^\sharp \mu^{-2\gamma-1} I_0 d^\gamma = 2^{2\gamma+2} \pi (\D_\gamma +1)^{-1} \frac{B((\D_\gamma+D_\omega)/2+1+\gamma, (\D_\gamma-D_\omega)/2+1+\gamma)}{B((\D_\gamma+D_\omega)/2+1, (\D_\gamma-D_\omega)/2+1)}.  
    \label{eq:funcrel}
\end{align}
Such a relation cannot be simplified due to the structure of the spectrum of these three operators, see \cite[Sec. 2.3.2]{Mishra2022}.

\subsection{From the Euclidean disk to constant curvature disks} \label{sec:CCDs}

All the previous results have analogues in the context of simple geodesic disks in constant curvature spaces. A model for such disks can be given by $\Dm_R = \{z\in \Cm,\ |z|^2\le R^2\}$, equipped with the metric $g_\kappa = (1+ \kappa|z|^2)^{-2} |dz|^2$. Then $(\Dm_R, g_\kappa)$ is a simple surface whenever $|\kappa|R^2<1$, of constant curvature $4\kappa$. The results present below are disseminated throughout \cite{Mishra2019,Mishra2022,Monard2019a}. Generally, we can derive the following: 

\begin{theorem}[Theorem 3 in \cite{Mishra2022}]\label{thm:CCD}
	Fix $\kappa\in \Rm$ and $R>0$ such that $|\kappa|R^2<1$. Let $I_0, I_0^\sharp$ the geodesic X-ray transform and backprojection operator associated with $(\Dm_R,g_\kappa)$. Then there exists $d\in C^\infty(\Dm_R)$ a boundary defining function for $\Dm_R$, and $\bt$, an $\alpha$-boundary defining function\footnote{See Sec. \ref{sec:phgonmwb} for a definition} for $\partial_+ S_{(\kappa)}\Dm_R$ such that for every $\gamma>-1$, the transform $I_0^\sharp \bt^{-2\gamma-1} I_0 d^\gamma$ is an isomorphism of $C^\infty(\Dm_R)$.
\end{theorem}

The main concept to exploit here is the {\em projective equivalence} between $(\Dm, g_0)$ and $(\Dm_R, g_\kappa)$, i..e. that for any $(R,\kappa)$ with $|\kappa|R^2<1$, there exists a diffeomorphism $\Phi_{\kappa,R}\colon \Dm_R\to \Dm$ such that the geodesics of $\Dm$ when equipped to either $g_0$ or $(\Phi_{\kappa,R}^{-1})^* g_\kappa$, are reparameterizations of one another. Such a fact induces intertwining diffeomorphisms between their X-ray transforms, and by further intertwining with the differential operators on $\Dm$ and $\partial_+ S\Dm$, this allows to find intertwining differential operators on $\Dm_R$ and $\partial_+ S_{(\kappa)}\Dm_R$, which ultimately drive Sobolev, or $C^\infty$-isomorphism mapping properties of the associated X-ray transform. 

In addition when $\gamma=0$, more is given in \cite[Theorem 1, Corollary 4]{Monard2019a}: for any such model (loosely written $(M,g)$ below), one may show that there exist a second-order degenerate elliptic operator $\L$ on $M$, a vector field $T$ on $\partial_+ SM$ and a non-vanishing weight function $w\in C^\infty(M)$ such that
\begin{align*}
	I_0 w \circ \L = (-T^2) \circ I_0 w \quad \text{ and }\quad \L (I_0^* I_0)^2 = Id \qquad \text{on } \quad C^\infty(M).
\end{align*}
Moreover, the Hilbert scales $\{\wtH^s(\Dm)\}_{s\in \Rm}$ and $\{H_{T,+}^s (\partial_+ S\Dm)\}_{s\in \Rm}$ defined in \eqref{eq:SobolevZernike} and \eqref{eq:scaledata}, and the operators $P_-,C_-$ all have analogues on $M$ and $\partial_+ SM$, that allow for range descriptions as sharp as \eqref{eq:rangecharac}.

Although not written in the literature, it is expected from the techniques in \cite{Mishra2022,Monard2019a} that for any $\gamma>-1$, there also exist second-order operators $\L_\gamma$ and ${\cal T}_\gamma$ on $M$ and $\partial_+ SM$, and a boundary defining function $d$ such that the intertwining property $I_0 d^\gamma \circ \L_\gamma = {\cal T}_\gamma \circ I_0 d^\gamma$ holds on $C^\infty(M)$.

%%%%%%%%%%%%%%%%%%%%%%%%%%%%%%%%%%%%%%%%%%%%%%%%%%%%%%%%%%%%%%%%%%%% X RAY
\section{Forward estimates for $I_0$ on convex, non-trapping manifolds} \label{sec:forward_Xray}

While the previous section worked in geometries with symmetries, where Singular Value Decomposition and special differential operators helped give quite a precise story, this section discuss results on the forward mapping properties of $I_0$ on convex, non-trapping manifolds, using more flexible toolboxes. We first recall in Sec. \ref{sec:cvx} the important features of convex, non-trapping manifolds, in particular the prominent role of the time-rescaling map (see Eq. \eqref{eq:upsilon}) for establishing the mapping properties that follow. We then discuss mapping properties of $I_0$, first on weighted $C^\infty$ spaces in Sec. \ref{sec:Cinftype}, then generalized to polyhomogeneous conormal spaces in Sec. \ref{sec:I0phg}. Finally, in Section \ref{sec:Honehalf}, we discuss instances in the literature of $H^{1/2}$-type spaces constructed on $\partial_+ SM$ for the purpose of capture mapping properties and stability estimates for $I_0$.

\subsection{Convex and non-trapping manifolds: important features}\label{sec:cvx}

We briefly discuss important functions and features that arise from the non-trapping and convexity conditions, see e.g. \cite[Sec. 3.1-3.3]{Paternain2023} and \cite[Sec. 6]{Monard2021} for detail. Define the first exit time function 
\begin{align}
	\tau(x,v) := \tau_+ (x,v),
	\label{eq:tau}
\end{align}
where the interval $[\tau_-(x,v), \tau_+ (x,v)]$ is defined as the {\em maximal} interval of existence of the geodesic $\varphi_t(x,v)$ in $SM$. We say that $(M,g)$ is {\em non-trapping} if $\sup_{SM} \tau(x,v) <\infty$. 

With $\nu_\x$ the inward-pointing unit normal at $\x\in \partial M$, we denote the second fundamental form 
\begin{align}
	\two_x(v,w) := - \langle\nabla_v \nu_\x, w\rangle_g, \qquad \x\in \partial M, \quad v,w\in T_\x \partial M.
	\label{eq:two}
\end{align}
We say that $\partial M$ is strictly (geodesically) convex if $\two_x$ is positive definite for all $\x\in \partial M$. An important consequence of convexity is that 
\begin{itemize}
	\item the function $\tau$ is continuous on $SM$ and smooth on $SM\backslash \partial_0 SM$, see \cite[Lemma 3.2.3]{Paternain2023}.
	\item the function $\tilde\tau(x,v) := \tau(x,v) - \tau(x,-v)$ is smooth on $SM$, see \cite[Lemma 3.2.11]{Paternain2023}. 
\end{itemize}
Along with the smoothness properties of the geodesic flow, this makes the footpoint map $F(x,v) := \varphi_{-\tau(\x,-v)}(\x,v)$ smooth on $(SM)^{int}$.

An important map first introduced in \cite[Sec. 6.3.1]{Monard2021} is the {\em time-rescaling map} $\Upsilon \colon \partial_+ SM \times [0,1]\to SM$ given by 
\begin{align}
	\Upsilon(x,v,u) := \varphi_{u\tau(x,v)}(x,v), \quad (x,v)\in \partial_+ SM,\quad u\in [0,1].
	\label{eq:upsilon}
\end{align}
It is a diffeomorphism on the interior, while it enjoys the following ``b-map'' property: if $d\in C^\infty(M)$ is a function equal to $d(\x) = \text{dist}(\x,\partial M)$ near $\partial M$ and non-vanishing on the interior, there exists a non-vanishing function $F\in C^\infty(\partial_+ SM\times [0,1], \Rm)$, such that 
\begin{align}
	d(\pi\circ \varphi_t(x,v)) = \tau^2(x,v) u(1-u) F(x,v,u).
	\label{eq:bmap}
\end{align}
In what follows, this function, via identity \eqref{eq:bmap}, will be crucial in understanding how boundary behavior transforms under application of $I_0$ or $I_0^\sharp$.

\subsection{Mapping properties on weighted $C^\infty$ type spaces}\label{sec:Cinftype}

On a convex, non-trapping Riemannian manifold $(M,g)$, the forward mapping properties of many variants of the X-ray transform (e.g. with attenuation) can be derived from the following general result. With $d$ a boundary defining function for $M$ (and hence for $SM$ as well, if we abuse notation $d$ for $d\circ \pi$). One may generally define the X-ray transform of $f\in d^{\gamma}C^\infty(SM)$ as  
\begin{align}
	I f(x,v) := \int_0^{\tau(x,v)} f(\varphi_t(x,v))\ dt, \qquad (x,v)\in \partial_+SM,
	\label{eq:I}
\end{align}
so that the definition of $I_0$ in \eqref{eq:Xray} is nothing but $I_0 = I\circ \pi^*$. 
%Fix $m,p$ two arbitrary integers. Let $d\in C^\infty(M)$ be a function equal to $d(\x) = \text{dist}(\x,\partial M)$ near $\partial M$ and non-vanishing on the interior. Given a weight $w\in C^\infty(SM, \Cm^{m\times p})$ and for $f\in C^\infty(M, \Cm^m)$ we define the weighted transform ${\cal I}^w \colon L^2(SM, \Cm^p) \to L^2(\partial_+ SM \to \Cm^m, \frac{\mu}{\tau}\ d\Sigma^{2d-2})$ as
%\begin{align}
%	{\cal I}^w f(x,v) := \int_0^{\tau(x,v)} w(\varphi_t(x,v)) f(\varphi_t(x,v))\ dt, \qquad (x,v)\in \partial_+SM.
%	\label{eq:weightedXray}
%\end{align}
Then we have the following forward mapping property (written in \cite[Proposition 6.13]{Monard2021} for weighted X-ray transforms there): 

\begin{proposition}[Proposition 6.13 in \cite{Monard2021}]\label{prop:fwdweightedXray}
	Let $(M,g)$ be a convex and non-trapping Riemannian manifold with boundary. Then for every $\gamma>-1$, the following mapping property holds
	\begin{align}
		I \colon d^{\gamma} C^\infty(SM) \to \tau^{2\gamma+1} C_{\alpha,+}^\infty(\partial_+ SM).
		\label{eq:fwdI}
	\end{align}
\end{proposition}

\begin{proof}[Proof of Proposition \ref{prop:fwdweightedXray} (sketch)]
	The proof follows from the properties of the time-rescaling map $\Upsilon$ \eqref{eq:upsilon}, since for $\gamma>-1$ and $f = d^\gamma h$ for some $h\in C^\infty(SM)$, one may write
	\begin{align*}
		If (x,v) = \int_0^{\tau(x,v)} f(\varphi_t(x,v))\ dt &\stackrel{(t = u\tau)}{=} \tau \int_0^1 (d(\Upsilon))^\gamma h(\Upsilon(x,v,u))\ du \\
		&\,\,\,\stackrel{\eqref{eq:bmap}}{=} \tau^{2\gamma+1} \int_0^1 h(\Upsilon(x,v,u)) F(x,v,u)^\gamma (u(1-u))^\gamma\ du. 
	\end{align*}
	One then concludes, using extensibility and symmetry properties of the functions $F$ and $\Upsilon$, by showing that the last integral defines a function in $C_{\alpha,+}^\infty(\partial_+ SM)$. 
\end{proof}

\subsection{On p(oly)h(omo)g(eneous) conormal spaces} \label{sec:I0phg}

The examples of spaces $d^\gamma C^\infty(M)$ and $\tau^{s} C^\infty_\alpha (\partial_+ SM)$ can be made more general. In \cite{Mazzeo2021}, the author and Mazzeo provide a more general framework where to formulate forward mapping properties of $I_0$ and $I_0^\sharp$ on convex, non-trapping manifolds. In order to discuss these results, we first need to recall some generalities on polyhomogeneous (phg) conormal spaces.

\subsubsection{Preliminaries: phg conormal functions on manifolds with boundary} \label{sec:phgonmwb}

Given $N^n$ a smooth manifold with boundary $\partial N$, we say that $x \in C^\infty(N, [0,\infty))$ is a {\bf boundary defining function} for $N$ if $\partial N = x^{-1} (0)$ and $dx \ne 0$ at each $p\in \partial N$. We define the Lie algebra of b-vector fields, $\V_b(N)$, to be the space of smooth vector fields which are unconstrained on the interior, and which are tangent to $\partial N$. Locally near a boundary point, where one may use local coordinates $(x, y_1, \dots, y_{n-1})$ with $x$ boundary defining and $(y_1,\dots,y_{n-1})$ coordinates on $\partial N$, we have 
    \begin{align}
	\V_b(N) = C^\infty - \text{span} \{x\partial_x, \partial_{y_1}, \dots, \partial_{y_{n-1}}\},
	\label{eq:Vb}
    \end{align}

    We now discuss a way of prescribing boundary behavior, for functions that are smooth in $N^{int}$, with which the mapping properties of $I_0$ and $I_0^\sharp$ can be described. Indeed, although $C^\infty(N)$ is the most obvious class of 'nice' functions, the examples from the introduction also suggest other spaces, such as $x^{-1/2} C^\infty(N)$. More generally, we want to define a class of Fr\'echet spaces with classical asymptotic expansions off of $\partial N$, involving countably many terms of the form $x^z \log^k x$, where $z\in \Cm$ should have real part bounded below, and $k\in \Nm_0$.  

Such spaces must first be thought of as subspaces of the following spaces of conormal functions, or functions that that enjoy iterated b-regularity in the sense below. Given $s\in \Rm$, we define
\begin{align}
    \A^s(N) = \left\{u\in C^\infty(N^{int}):\ V_1\cdots V_k u \in x^s L^\infty(N), \text{ for all } k\in \Nm_0 \text{ and } V_1, \dots, V_k \in \V_b\right\}.
    \label{eq:As}
\end{align}
Such spaces still display too large a wealth of boundary behavior (see discussion in \cite[Sec. 2.1]{Mazzeo2021}), and we will pick subspaces of the form \eqref{eq:Aphg} below, as follows. 

\begin{definition}[Index set]\label{def:idxset}
	A {\bf $C^\infty$-index set} $E$ is a subset of $\Cm\times \Nm_0$ such that 
	\begin{enumerate}
		\item[(a)] For every $s\in \Rm$ the set $E_{\le s} = \{(s,k)\in E\colon \text{Re}(z)\le s\}$ is finite. 
		\item[(b)] If $(z,k)\in E$, then $(z,\ell)\in E$ for every $\ell= 0, 1, \dots, k-1$. 
		\item[(c)] If $(z,k)\in E$, then $(z+1,k)\in E$. 
	\end{enumerate}	
\end{definition}
Condition (a) implies the existence of a ``smallest'' $s$ (denoted $\inf E$) such that $\text{Re}(z) \ge s$ for all $(z,k)\in E$. Now given an index set $E$ and $s\le \inf E$, we can define 
\begin{align}
    \A_{\mathrm{phg}}^E (N) := \left\{ u\in \A^s(N)\colon u \sim \sum_{(z,k)\in E} u_{z,k}(y) x^z (\log x)^k, \quad u_{z,k}\in C^\infty(N) \right\}. 
    \label{eq:Aphg}
\end{align}

As examples, $C^\infty(N)$ corresponds to $E = \Nm_0\times \{0\}$, while $x^{-1/2} C^\infty(N)$ corresponds to $E = (\Nm_0-1/2) \times \{0\}$. One may see that the conditions (a), (b), (c) make the definition of $\A_{\mathrm{phg}}^E (N)$ invariant under change of boundary defining function and multiplication by a function in $C^\infty(N)$. 

\paragraph{Particularizing to a Riemannian manifold and its inward-pointing boundary.} 

On a Riemannian manifold with boundary $(M,g)$, a good choice of boundary defining function is to extend the function $\x\mapsto d_g(\x, \partial M)$ (usually well-defined and smooth on a collar neighborhood of $\partial M$) into a globally well-defined and smooth function on $M$, and all other definitions essentially follow \eqref{eq:As} and \eqref{eq:Aphg}.

On to the manifold $\partial_+ SM$, this manifold carries a smooth and geometrically relevant involution, the {\em antipodal scattering relation} 
\begin{align*}
    \alpha_a(x,v) := \varphi_{\tau(x,-v)} (x,-v), \qquad (x,v)\in \partial_{+} SM,
\end{align*}
i.e., the composition of the scattering relation and the antipodal map. As the range of the X-ray transforms is only made of functions with are even with respect to $\alpha_a$, accurately capturing the mapping properties that follow requires to define index sets and boundary defining functions which capture symmetries with respect to $\alpha$ and $\alpha_a$. To that end, following \cite[Sec. 4.3]{Mazzeo2021}, we will say that a boundary defining function $\bt$ on $\partial_+ SM$ is an {\bf $\alpha$-boundary defining function} ($\alpha$-bdf, for short) if $\bt\in C_{\alpha,-}^\infty(\partial_+ SM)$. If $(M,g)$ has convex boundary, then three good examples of $\alpha$-bdf's are given by: (i) the function $\tau(x,v)$ on $\partial_+ SM$; (ii) the $\partial M$-geodesic distance from $x\in \partial M$ to $\pi(\alpha(x,v))$ (see \cite[Sec. 4.1]{Mazzeo2021}); (iii) the function $\mu(x,v) + \mu(\alpha_a(x,v))$. 

Further, we define a {\bf $C_\alpha^\infty$-index set} on $\partial_+ SM$ to be a set $E\subset \Cm\times \Nm_0$ as in Def. \ref{def:idxset}, with (c) replaced by 
\begin{enumerate}
    \item[(c')] If $(z,k)\in E$, then $(z+2,k)\in E$.
\end{enumerate}
Then for $E$ a $C_\alpha^\infty$-index set, one may define $\A_{\mathrm{phg}}^E (\partial_+ SM)$ as in \eqref{eq:Aphg}, and such a space is invariant under change of $\alpha$-bdf and multiplication by a function in $C_{\alpha,+}^\infty(\partial_+ SM)$. Examples of such spaces are $C_{\alpha,+}^\infty(\partial_+ SM)$ (when $E = 2\Nm_0\times \{0\}$) and $C_{\alpha,-}^\infty(\partial_+ SM)$ (when $E = (2\Nm_0+1)\times \{0\}$), defined in \eqref{eq:Calphapm}.

\paragraph{Coordinate systems near the boundaries of $M$ and $\partial_+ SM$.}

To carry out the analysis below, we also need coordinate systems amenable to explicit computations for $M$ and $\partial_+ SM$ near their respective boundaries. On $M$, we use traditional Riemannian boundary normal coordinates $(\rho,y)\in [0,\varepsilon]\times \partial M$, where $\rho$ denotes the $g$-distance to the boundary. On $\partial_+ SM$, we use $\tau$ as $\alpha$-bdf for $\partial_0 SM$, completed with a $\partial_0 SM$-valued, $(2d-2)$-dimensional coordinate $\omega$ which, if $(\x,v)\in \partial_+ SM$ is near $\partial_0 SM$, $\omega$ denotes the geodesic midpoint and tangent vector in $S(\partial M) \approx \partial_0 SM$ of the $g|_{\partial M}$-geodesic joining $\x$ to $\pi (\alpha(\x,v))$. That $\omega$ is well-defined is done in \cite[Lemma 4.1]{Mazzeo2021}. What makes the coordinates $(\tau,\omega)$ desireable is that they have good symmetries w.r.t. the scattering relation: $\tau\in C_{\alpha,-}^\infty(\partial_+ SM)$ while $\omega\in C^\infty_{\alpha,+}(\partial_+ SM, \partial_0 SM)$, so that $C_{\alpha}^\infty$-expansions of phg functions near $\partial_0 SM$ can be written as asymptotic sums involving terms of the form $\tau^z (\log\tau)^k a_{z,k}(\omega)$, where $(z,k)\in \Cm\times \Nm_0$ and $a_{z,k}\in C^\infty(\partial_0 SM)$. Moreover, they are geometrically relevant and as such, they enable the computation of volume factors in terms of intrinsic geometric quantities (see \cite[Sec. 5.5]{Mazzeo2021}).

\begin{remark}
	The two-dimensional case can be refined slightly: on simple surfaces, the construction above can be upgraded into a global chart $(w,p)$ of $\partial_+ SM$, see \cite[Sec. 3.3.2]{Monard2021a}. There, such coordinates give an analogue of parallel-beam coordinates in the Euclidean setting, which are relevant when tackling practical questions such as sampling: given that $\partial_+ SM$ is a 2-1 covering of the space of unoriented geodesics where $I_0$ is defined, discrete uniform sampling rates in $(w,p)$ will be equally constrained at both representative of each geodesic, a feature which does not occur for fan-beam coordinates.
\end{remark}

\subsubsection{Mapping properties of $I_0$ on phg spaces}

The refined mapping properties of $I_0$ (or more generally, $I$ defined in \eqref{eq:I}) on phg-conormal spaces is formulated in \cite{Mazzeo2021}. In words: the X-ray transform maps phg-conormal spaces on $M$ to other phg-conormal spaces on $\partial_+ SM$, one can tell how expansions off of $\partial M$ are transformed into expansions off of $\partial_0 SM$, with an explicit expression for the most singular (intrinsic) coefficients in some coordinate systems. 

\begin{theorem}[Theorem 2.1 in \cite{Mazzeo2021}] \label{thm:Calphamapping}
Let $(M,g)$ be convex, non-trapping, and let $\phi\in C^\infty(SM)$. Suppose $f \in \A^E_{\phg}(SM)$ for some index set $E$ with 
$s = \inf(E)>-1$. Then $I^\phi f \in  \A^F_{\phg}(\partial_+ SM)$, where $F$ is a $C^\infty_\alpha$-index set contained in 
(and possibly equal to) $\{ (2z+1,\ell), \ (z,\ell)\in E \}$.

More precisely, fix $\gamma\in \Cm$, $k\in \Nm_0$, $h\in C^\infty(\partial SM)$ and $\chi\in C_c^\infty([0,\infty))$ a cutoff function equal to $1$ near $0$.
Using coordinates $(\rho,\xi)\in [0,\varepsilon)\times \partial SM$ on $SM$ near $\partial SM$ and $(\tau,\omega)$ on $\partial_+ SM$ near 
$\partial_0 SM$ with $\omega = (y,w)\in \partial_0 SM$, as defined in the previous section, and denoting by $\tau$ the geodesic length, we have 
\begin{align*}
I^\phi [\rho^\gamma \log^k \rho\ h(\xi)] (\tau,\omega) \sim  \tau^{2\gamma+1} \sum_{p\ge 0 \atop 0 \le \ell \le k} \tau^{2p} (\log\tau)^\ell 
a_{2\gamma+1+p,\ell}(\omega).
\end{align*}
If $\rho(x) = d_g(x,\partial M)$ near $\partial M$, the coefficient in front of the most singular term equals 
    \begin{align}
	a_{2\gamma+1,k}(\omega) = 2^k h(\omega) \phi(\omega) \two(\omega)^\gamma\ B(\gamma+1,\gamma+1),
	\label{eq:I0BottomTerm}
    \end{align}
where $\two(\omega) := \two_y(w,w)$ is the second fundamental form, and $B$ is the Beta function.
\end{theorem}

\begin{proof}[Proof of Theorem \ref{thm:Calphamapping} (sketch)]
	Property \eqref{eq:bmap} alone immediately allows us to understand how terms of a phg expansion of $f$ off of $\partial M$ translate into terms of the phg expansion of $I_0 f$ off of $\partial_0 SM$, as can be seen from the following calculation: a tubular neighborhood of $\partial M$ may look like $[0,\varepsilon]_\rho \times (\partial M)_y$, so that a general term of the boundary expansion of $f\in \A^s(M)$ off of $\partial M$ looks like $\rho^z (\log \rho)^ka(y)$ for some $(z,k)\in \Cm\times \Nm_0$ and $a\in C^\infty(\partial M)$, and it is enough to consider each such term separately by linearity of $I_0$. In the case $k=0$, the computation looks like 
	\begin{align*}
		I_0 [\rho^z a(y)] &= \int_0^\tau \rho^z(\gamma_{x,v}(t)) a(y(\gamma_{x,v}(t))) \ dt \\
		&\!\!\stackrel{t = \tau u}{=} \tau(x,v) \int_0^1 \rho^z (\Upsilon(x,v,u)) a(y(\Upsilon(x,v,u)))\ du \\
		&\!\!\stackrel{\eqref{eq:bmap}}{=} \tau(x,v) \int_0^1 \tau(x,v)^{2z} (u(1-u))^z F^z(z,v,u) a(y(\Upsilon(x,v,u)))\ du \\
		&= \tau^{2z+1} (x,v) \underbrace{\int_0^1 F^z(z,v,u) a(y(\Upsilon(x,v,u)))\ (u(1-u))^z\ du}_{b(x,v;z)}.
	\end{align*}
	We immediately see that the ``integrability condition'' $\text{Re}(z) >-1$ is required to make $b(x,v;z)$ well defined. Further, one may show that $b(\cdot; z) \in C_{\alpha,+}^\infty(\partial_+ SM)$, and as such, admits an asymptotic expansion near $\partial_0 SM$ of the form $b \sim \sum_{p\ge 0} \tau^{2p} b_p (\omega;z)$, with 
	\begin{align*}
		b_0(\omega;z) = \lim_{\tau\to 0} b(\tau,\omega;z) = a(\pi(\omega)) B(z+1,z+1) \two (\omega)^z,
	\end{align*}
	where $B$ denotes the Beta function and $\two(\omega) := \two_y(w,w)$ is the second fundamental form defined in \eqref{eq:two}, evaluated at $\omega = (y,w)$ viewed as an element of $S(\partial M)$. 
	
	As a consequence, we immediately see that $I_0 [\rho^z a(y)]$ admits a $C_\alpha^\infty$ expansion with powers $\{\tau^{2z+1+2p}\}_{p\in \Nm_0}$, with most singular coefficient expressible in terms of intrinsic geometric quantities. 
	The inclusion of $\log$ terms in the boundary expansion can be found in \cite{Mazzeo2021}. 	
\end{proof}

\subsection{$H^{1/2}$ spaces on $\partial_+ SM$ and sharp stability estimates} \label{sec:Honehalf}

On general surfaces, the scales of spaces \eqref{eq:SobolevZernike}, \eqref{eq:scaledata} and \eqref{eq:scale_iso} have not yet been defined, however other Sobolev spaces have been defined on $\partial_+ SM$ where to formulate sharp $L^2-H^{1/2}$ stability estimates. On a general simple surface, the method of Pestov's energy identities allows to derive an $L^2-H^1$ stability estimate of the form 
\begin{align*}
	\|f\|_{L^2(M)} \le C \|I_0 f\|_{H^1(\partial_+ SM)},
\end{align*}
however it does not capture the $1/2$ smoothing property of $I_0$, which in itself requires the design of an appropriate $H^{1/2} (\partial_+ SM)$ space. The latter space can also be defined in terms of regularity with respect to a subset of all variables, and this requires finding the appropriate ones. We now discuss the previous literature addressing these issues. 

\subsubsection{A criterion for anisotropic Sobolev spaces on $\partial_+ SM$}\label{sec:ASstab}

In \cite{Assylbekov2018}, Assylbekov and Stefanov proposed a construction of an anisotropic $\overline{H}_{\partial_+ SM}^{1/2}$ space, and gave an important criterion on the choice of variables (with respect to which one should measure regularity) that make it possible to prove an $L^2(M) \to \overline{H}_{\partial_+ SM}^{1/2}$ stability estimate for X-ray transforms.

The construction of $\overline{H}_{\partial_+ SM}^{1/2}$ goes as follows:

\smallskip
$\bullet$ locally near $(x_0,v_0)\in \partial_+ SM$, fix $k\in \{0,1,\dots,2n-2\}$ and split the $2n-2$ coordinates into $(y',y'')$ with $y' = (y^1,\dots,y^k)$ and $y'' = (y^{k+1},\dots,y^{2n-2})$ satisfying the important requirement (see \cite[Eq. (3.2)]{Assylbekov2018}):
\begin{align}
	\forall y'',\quad \text{the map } (t,y')\mapsto \gamma_{(y',y'')}(t) \text{ is a submersion}.
	\label{eq:submersionCondition}
\end{align}
This condition forces in particular $k\ge n-1$. Then define an $\overline{H}^s$ space using $y'$-derivatives only: for $h$ supported near $(x_0,y_0)$, as in
\begin{align*}
	\|h\|_{\overline{H}^s}^2 = \int (1+|\xi'|^2)^s |{\cal F}_{y'\to \xi'} h (\xi',y'')|^2\ d\xi'\ dy'',
\end{align*}
where $\cal F$ denotes the Fourier transform.

\smallskip
$\bullet$ To make this construction global on $\partial_+ SM$, consider $M_1$ a simple neighborhood of $M$, then via the footpoint map $SM_1\to \partial_+ SM_1$, one thinks of $\partial_+ SM$ as a compact subset of the manifold $(\partial_+ SM_1)^{int}$, which can be covered by finitely many charts as above. Using a partition of unity $\{\chi_j\}_j$, one may complete the above into a global norm
\begin{align*}
	\|h\|_{\overline{H}^s_{\partial_+ SM}}^2 = \sum_j \|\chi_j h\|^2_{\overline{H}^s} .
\end{align*}

For such spaces, the stability estimate reads as follows (the statement in \cite{Assylbekov2018} also covers stability estimates for one-forms and symmetric two-tensors):

\begin{theorem}[Theorem 1.1.(a) in \cite{Assylbekov2018}]\label{thm:ASstab}
	Let $(M,g)$ be a simple manifold and let $\overline{H}_{\partial_+ SM}^{1/2}$ be constructed as above. Then if $\kappa\in C^\infty(SM)$ is a smooth weight such that the weighted X-ray transform\footnote{In the notation of this article, $I_{0,\kappa} f := I(\kappa \cdot f\circ \pi)$ with $I$ defined in \eqref{eq:I}.} $I_{0,\kappa}$ is injective, there exists a constant $C$ such that
	\begin{align}
		\|f\|_{L^2(M)} / C \le \|I_{0,\kappa} f\|_{\overline{H}_{\partial_+ SM}^{1/2}} \le C \|f\|_{L^2(M)}, \qquad \forall f\in L^2(M).
		\label{eq:ASstab}
	\end{align}
\end{theorem}

The requirement \eqref{eq:submersionCondition} is not needed for the construction of the space, but it enters the proof of Theorem \ref{thm:ASstab} in a crucial way, in that it is equivalent to making sure that the derivatives $(\partial_{y^1}, \dots, \partial_{y^k})$ are microlocally elliptic on the image of $T^*M \backslash 0$ under the canonical relation of $I_{0,\kappa}$ (viewed as a Fourier Integral Operator). When \eqref{eq:submersionCondition} is satisfied, the stability estimate follows from a constructing a local parametrix, and exploiting the assumed injectivity of $I_{0,\kappa}$ to remove error terms in smoother norms that parametrix constructions usually produce.

In \cite[Sec. 3]{Assylbekov2018}, the authors also give examples of admissible variables $y'$. In the case $n=2$ and $k=1$, they explain how one may recover the well-known fact that in the Euclidean case, $L^2-H^{1/2}$ stability estimates can be formulated with a scale on the right side defined in terms of Fourier decay only with respect to one variable, see \cite[Sec. II.5]{Natterer2001}. This is also the 'anisotropic' nature of the scale \eqref{eq:scaledata}.

\subsubsection{A sharp stability estimate for $I_0$ on non-negatively curved surfaces}\label{sec:PSstab}

In \cite{Paternain2018} Paternain and Salo defined a $H^{1/2}$-type of Sobolev space on $\partial_+ SM$ for any simple surface with non-positive Gaussian curvature, where to formulate a sharp $L^2 - H^{1/2}$ stability estimate for the X-ray transform. We will briefly describe the idea and refer the interested reader to \cite[Sec. 7.2]{Paternain2023} or \cite{Paternain2018} for detail. 

Specifically, a frame of $\partial SM$ is given by the vertical vector field $V$ (generator of the circle action on the fibers about boundary points) and the tangential vector field $T$ (the Sasaki-horizontal derivative along an oriented tangent vector to $\partial M$), see \cite[Def. 4.5.2]{Paternain2023}. With $L^2(\partial SM)$ defined with respect to the natural Sasaki area form $d\Sigma^2$ on $\partial SM$, one may define the space\footnote{Here we have chosen to keep the $H^s_T$ notation introduced in \cite{Paternain2023}, though the spaces differ from \eqref{eq:scaledata}, as the vector fields $T$ also differ from the one introduced above \eqref{eq:Lzero}.} $H^1_T (\partial_+ SM)$ to be the completion of $C^\infty(\partial SM)$ for the norm
\begin{align*}
	\|w\|^2_{H_T^1 (\partial SM)} = \|w\|_{L^2(\partial SM)}^2 + \|Tw\|_{L^2(\partial SM)}^2,
\end{align*}
and the space $H_T^{1/2}(\partial_+ SM)$ to be the complex interpolation space between $L^2(\partial_+ SM)$ and $H^1_T (\partial_+ SM)$. Then the Paternain-Salo stability estimate reads:  

\begin{theorem}[Theorem 7.2.1 in \cite{Paternain2023}]\label{thm:PSstab}
	Let $(M,g)$ be a compact simple surface with non-positive Gaussian curvature. Then 
	\begin{align}
		\|f\|_{L^2(M)} \le \frac{1}{\sqrt{2\pi}} \|I_0 f\|_{H_T^{1/2}(\partial_+ SM)}, \qquad f\in C^\infty(M).
		\label{eq:PSstab}
	\end{align}
\end{theorem}

The main idea is based on a refinement of the Pestov, or Guillemin-Kazhdan identities (energy identities derived by integration by parts on $SM$). Let $X_{(\x,v)} := \frac{d}{dt}|_{t=0} \varphi_t(\x,v)$ be the geodesic vector field on $SM$. Given $f\in C^\infty(SM)$, if $u^f$ denotes the unique solution to the transport equation $Xu = -f$ on $SM$ with boundary condition $u|_{\partial_- SM} = 0$ giving $u|_{\partial_+ SM} = I_0 f$, then the well-known Pestov-identity with boundary term generates the estimate $\|f\|^{2}_{L^2(SM)} \le - (T u^f, Vu^f)_{\partial SM}$. This is, in spirit, an $L^2-H^1$ stability estimate, and as such does not capture the expected $1/2$-smoothing property. By careful use of the Guillemin-Kazhdan identity and crucially exploiting non-positive curvature, the authors are able to derive the sharper stability estimate (see \cite[Lemma 7.2.4]{Paternain2023})
\begin{align*}
	\|f\|_{L^2(SM)}^2 \le \left( T u^f, H u^f \right)_{\partial SM}, \qquad f\in C^\infty(M), 
\end{align*}
where $H$ denotes the fiberwise Hilbert transform, a zero-th order operator on the fibers of $\partial SM$. Theorem \ref{thm:PSstab} follows then by showing that this last right-hand side is equivalent to a squared $H^{1/2}_T(\partial_+ SM)$ norm (see \cite[Lemma 7.2.5]{Paternain2023}) of $I_0 f$.

The same authors also generalize Theorem \ref{thm:PSstab} to tensor fields and manifolds of dimension greater than two, see \cite[Sec. 7.3-7.5]{Paternain2023}.

%%%%%%%%%%%%%%%%%%%%%%%%%%%%%%%%%%%%%%%%%%%%%%%%%%%%%%%%%%%%%%%% BACKPROJECTION
\section{Forward estimates for $I_0^\sharp$ on convex, non-trapping manifolds} \label{sec:forward_backproj}

We now turn to the study of $I_0^\sharp$. We should preface this with the comment that the study of boundary behavior mapping under $I_0$ is somewhat simpler than that of $I_0^\sharp$ for the following reason: {\em in convex geometries, the boundary behavior of $I_0 f$ near $\partial_0 SM$ is only dictated by the boundary behavior of $f$ near $\partial M$}, in the sense that if $f$ has compact support in $M^{int}$, then $I_0 f$ has compact support in $(\partial_+ SM)^{int}$. This is no longer true for $I_0^\sharp$: the image of a compactly supported function in $(\partial_+ SM)^{int}$ under $I_0^\sharp$ is, in general, supported all the way to $\partial M$.

In this chapter, we will first describe Pestov and Uhlmann's original result on a good domain of definition for $I_0^\sharp$ in Sec. \ref{sec:PUresult}. We will then discuss the generalization to phg-conormal spaces in Sec. \ref{sec:bfib}. Finally, we will discuss a result of a similar flavor in Sec. \ref{sec:Katsevich}, applied to the adjoint of the classical Radon transform on $\Rm^2$. 

%One way to provide a partial answer to the mapping properties of both operators $I_0$ and $I_0^\sharp$, is to generalize the double fibration $(\partial_+ SM)^{int} \stackrel{F}{\longleftarrow} (SM)^{int} \stackrel{\pi}{\longrightarrow} M^{int}$ into one that holds up to the boundary, with polyhomogeneous behavior. This is the context of celebrated {\em b-fibrations between manifolds with corners}, introduced by Melrose and forming the cornerstone of geometric microlocal analysis and the construction of pseudodifferential calculi, see e.g. \cite{Melrose1992,Grieser2001}.

\subsection{The Pestov-Uhlmann result} \label{sec:PUresult} 

We first describe the original result of Pestov and Uhlmann, regarding the construction of a natural domain of definition for $I_0^\sharp$ to land into $C^\infty(M)$. Recall the definition $C_{\alpha,+}^\infty(\partial_+ SM) = \{w\in C^\infty(\partial_+ SM),\ A_+ w\in C^\infty(\partial SM)\}$ from \eqref{eq:Calphapm}, then we have the following characterization
\footnote{Notationally, the statement differs from the original one in two ways: first, the space $C_{\alpha,+}^\infty$ is originally written as $C_\alpha^\infty$; second, it is originally defined via \eqref{eq:PUcharac2}, then defined to equal the expression in \eqref{eq:Calphapm}. The reason for the switch is that the current definition admits an analogous definition with $+$ replaced by $-$, while the expression in the right-hand side of \eqref{eq:PUcharac2} does not.}:

\begin{theorem}[Lemma 1.1 in \cite{Pestov2005}]	\label{thm:PU}
	Let $(M,g)$ be a compact non-trapping manifold with strictly convex boundary. Then 
	\begin{align}
		C_{\alpha,+}^\infty (\partial_+ SM) := \{w \in C^\infty(\partial_+ SM): w\circ F \in C^\infty(SM)\}.
		\label{eq:PUcharac2}
	\end{align}
\end{theorem} 
As a result, for $w\in C_{\alpha,+}^\infty (\partial_+ SM)$, since $I_0^\sharp w$ is the fiber-average of $w\circ F$ which is smooth on $SM$ by virtue of Theorem \ref{thm:PU} we naturally have that 
\begin{align}
	I_0^\sharp (C_{\alpha,+}^\infty(\partial_+ SM))\subset C^\infty(M).
	\label{eq:I0sharpPU}
\end{align}
%\begin{remark} 
%	Using Theorem \ref{thm:I0starIDX} below, applying $I_0^\sharp$ na\"ively to a function $w\in C^\infty(\partial_+ SM)$ would land into the larger space $C^\infty(M) + d\log d C^\infty(M)$. 	
%\end{remark}

\begin{proof}[Proof of Theorem \ref{thm:PU} (sketch)] There are two proofs, whose detail in full glory can be found in Sections 5.1 and 5.2 of \cite{Paternain2023}, respectively.

	Pestov and Uhlmann's original proof is based on the theory of fold maps, see \cite[Sec. 4]{Pestov2005} or \cite[Sec. 5.2]{Paternain2023}: $M$ can be embedded in the interior of a simple manifold with boundary $N$, then the footpoint map of $N$, $F^N \colon SN\to \partial_+ SN$ (defined as in \eqref{eq:footpoint}), when restricted to $\partial SM$, is proved in \cite[Theorem 4.1]{Pestov2005} to be a Whitney fold\footnote{in the sense of e.g. \cite[Def. 5.3.1]{Paternain2023}} at every point of $\partial_0 SM$, with associated involution the scattering relation $\alpha$ on $\partial SM$. With this context in place, the characterization \eqref{eq:PUcharac2} follows from \cite[Th. C.4.4]{Hoermander2007}. 

	A second proof, also based on ideas from fold theory but giving more refined information on the singular behavior of the exit time function $\tau$ near $\partial_+ SM$, is given in \cite[Sec. 5.1]{Paternain2023}.
\end{proof}

\subsection{Forward mapping properties on phg-conormal spaces} \label{sec:bfib}

We now present a generalization of \eqref{eq:I0sharpPU} to phg-conormal spaces appearing in \cite[Theorem 2.2]{Mazzeo2021}. Recall the generalities on phg-conormal spaces introduced in Section \ref{sec:phgonmwb}. Then \cite[Theorem 2.2]{Mazzeo2021} states, in words: $I_0^\sharp$ maps $C_\alpha^\infty$-phg spaces on $\partial_+ SM$ to phg spaces on $M$; the action of $I_0^\sharp$ displays more complex phenomena than that of $I_0$ (seen in Theorem \ref{thm:Calphamapping}), such as the creation, or annihilation, of $\log$ terms in the resulting expansions; in specific coordinate systems, one can compute the singular coefficients in the expansion (the coefficient of the most singular term is most relevant as it is geometrically invariant). More specifically:

\begin{theorem}[Theorem 2.2 in \cite{Mazzeo2021}]\label{thm:I0starIDX}
    Fix $\gamma\in \Cm$, $k\in \Nm_0$ and $a\in C^\infty(\partial_0 SM)$. Then $I_0^\sharp (a(\omega) \tau^\gamma \log^k \tau \chi(\tau))$ is polyhomogeneous, with index set contained in (and possibly equal to) the union $\Nm_0\times \{0\} \cup E_{\gamma,k}$, where 
    \begin{align}
	    E_{\gamma,k} := \left\{
	    \begin{array}{ll}
		    (\frac{\gamma+1}{2} + \Nm_0) \times \{0, \dots, k-1\}, & \gamma\in 2\Nm_0, \\
		    (\frac{\gamma+1}{2} + \Nm_0) \times \{0, \dots, k\} \cup ( \Nm_0 \cap (\frac{\gamma+1}{2} + \Nm_0) ) \times \{k+1\}, & \gamma\in 2\Zm+1, \\
		    (\frac{\gamma+1}{2} + \Nm_0) \times \{0, \dots, k\}, & \text{otherwise},
	    \end{array}
	    \right.
	    \label{eq:Egamk}
    \end{align}
    where in the first case, $E_{\gamma,k} = \emptyset$ if $k=0$.
    In coordinates $(\rho,y)\in [0,\varepsilon)\times \partial M$ on $M$ near $\partial M$ 
    with $\rho(x) = d_g(x,\partial M)$, we conclude that
    \begin{align*}
	    I_0^\sharp  (a(\omega) \tau^\gamma \log^k \tau \chi(\tau)) \sim \sum_{(z,\ell)\in \Nm_0 \cup E_{\gamma,k}} a_{z,\ell}(y) \rho^z \log^\ell \rho\ dA(y),
	    \end{align*}
	    where the coeffient of lowest order in the part of the expansion corresponding to $E_{\gamma,k}$ is given by
	    \begin{align*}
		a_{z_0,\ell_0}(y) = c_{z_0,\ell_0} \int_{S_y (\partial M)} a(y,w) \two_y(w)^{\frac{-\gamma+1}{2}}\ dS_y(w), \qquad \omega = (y,w)\in \partial_0 SM,
	    \end{align*}
	    with the following expressions: 
	    \begin{center}
		    \begin{tabular}{c|c|c}
			    case & $(z_0,\ell_0)$ & $c_{z_0,\ell_0}$ \\
			    \hline
			    \hline
			    $\gamma = 2m, m\in \Nm_0$ & $(m+1/2,k-1)$ & $(2\pi k)(4^{2m+2})/2^{(k+3)(2m+2)}$ \\
			    $\gamma = 2m-1, m\in \Nm_0$ & $(m,k+1)$ & $-\binom{2m}{m}/((k+1) 2^{k+2})$ \\
			    $\gamma\notin \Nm_0-1$ & $(\frac{\gamma+1}{2},k)$ & ${2^{-(k+3)}} B\left(-\frac{\gamma+1}{2}, -\frac{\gamma+1}{2} \right)$
		    \end{tabular}
	    \end{center}
%	    Here 
%	    \begin{equation}
%		\begin{cases}
%		    \gamma=2m \Rightarrow (z_0,\ell_0) = (m+1/2,k-1), \  &c_{z_0,\ell_0} = (2\pi k)(4^{2m+2})/2^{(k+3)(2m+2)}, \ \ m\in\Nm_0,  \\[0.5ex]
%		    \gamma=2m-1 \Rightarrow (z_0,\ell_0) = (m, k+1),  & c_{z_0,\ell_0} = -\binom{2m}{m}/((k+1) 2^{k+2}), \ \  m\in\Nm_0, \\[0.5ex]
%		    \gamma\notin \Nm_0-1 \Rightarrow  (z_0,\ell_0) = (\frac{\gamma+1}{2},k), & c_{z_0,\ell_0} = {2^{-(k+3)}} B\left(-\frac{\gamma+1}{2}, -\frac{\gamma+1}{2} \right).  
%		\end{cases}
%		\label{eq:Agk}
%	    \end{equation}
\end{theorem}

In the next two paragraphs, we first explain why mapping properties between phg spaces can be expected, by viewing $I_0$ and $I_0^\sharp$ as involving pushforward and pullback by b-fibrations and invoking Melrose's celebrated PushForward and PullBack theorems, with the caveat that a direct applications produces an overestimation of the image index set in the study of $I_0^\sharp$. We then give a scheme of proof for Theorem \ref{thm:I0starIDX}, in particular explaining situations where the creation or annihilation of log terms can occur. 

%An important map for what follows is the introduction of the time-rescaling map
%\begin{align}
%    \Upsilon \colon \partial_+ SM\times [0,1] \ni (x,v,u) \mapsto \varphi_{u\tau(x,v)} (x,v)\in SM. 
%    \label{eq:Upsilon}
%\end{align}
%A key fact about $\Upsilon$ is the following property (see \cite[Sec. 6.3.1]{Monard2021}): there exists a smooth, positive function $F\colon \partial_+ SM \times [0,1]\to \Rm$ such that
%\begin{align}
%	\rho (\pi\circ \Upsilon (x,v,u)) = \tau^2 (x,v)u(1-u) F(x,v,u), \qquad (x,v,u)\in \partial_+ SM \times [0,1], 
%	\label{eq:rhoprop}
%\end{align}
%where $\rho(x) = \text{dist} (x,\partial M)$ near $\partial M$. %What is key about \eqref{eq:rhoprop} is the fact that the map $\pi\circ\Upsilon \colon \partial_+ SM\times [0,1]\to M$ has a well-understood boundary behavior: namely, pulling back a function on $M$ into a function on $\partial_+ SM \times [0,1]$ using $\pi\circ\Upsilon$ will turn first-order vanishing at $\partial M$ into product-type vanishing with specified order at the boundary hypersurfaces of $\partial_+ SM \times [0,1]$

\paragraph{The double b-fibration picture and Melrose's pushforward and pullback theorems.}

The celebrated PushForward and PullBack theorems of Melrose \cite{Melrose1992} are the cornerstone of the construction of exotic pseudo-differential algebras on manifolds with boundary, allowing for the construction of pseudodifferential parametrices of differential operators whose ellipticity is defined with respect to a distinguished 'asymptotic' geometric model at the boundary. Examples include the b-calculus \cite{Grieser2001}, adapted to asymptotically cylindrical geometries, and the 0- (more generally, edge-) calculus \cite{Mazzeo1991}, adapted to asymptotically hyperbolic geometries. Loosely speaking, many operations between distributions (e.g. operator composition) can be viewed as compositions of pushforward and pullback operations by {\em b-fibrations}\footnote{i.e., maps between manifolds with corners which are fibrations in the interior, and with specific behavior at the boundary hypersurfaces encoded by a matrix of scalars indexed over the bhf's of the domain and the target manifolds; see \cite[Sec. 2.1]{Mazzeo2021}.}, and Melrose's theorems help predict which phg-conormal behavior gets generated at each step.   

In the case of $I_0$ and $I_0^\sharp$ on convex, non-trapping manifolds, theres is a way to exploit this machinery as follows. Heuristically, the footpoint map $F\colon SM\to \partial_+ SM$ defined in \eqref{eq:footpoint} is just shy of being a fibration of $SM$ by geodesics, in that such geodesic segments collapse to single points as $(\x,v)$ approaches $\partial_0 SM$. However, the degeneracy it has at $\partial_0 SM$ can be resolved by the map $\Upsilon$ introduced in \eqref{eq:upsilon}, in the way that the dashed arrows in the diagram below
%\begin{figure}[htpb]
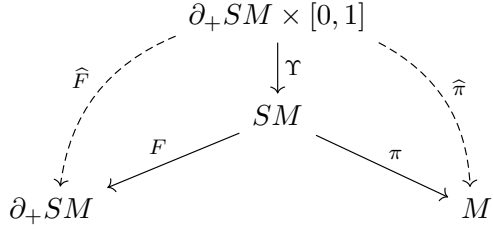
\begin{center}
	\begin{tikzcd}[]
    & \partial_+ SM \times [0,1]  \arrow{d}{\Upsilon} \arrow[dashed, bend left]{ddr}{\widehat{\pi}} \arrow[dashed, bend right]{ddl}[swap]{\widehat{F}} & \\  &  SM \arrow{dl}[swap]{F}\arrow{dr}{\pi} & \\
\partial_+ SM && M
\end{tikzcd}	
\end{center}    
%\caption{}%Desingularization of the double fibration into a double b-fibration via the map $\Upsilon$}
%\label{fig:doublefib}
%\end{figure}
make up a double b-fibration between manifolds-with-corners, with the following geometric data

\begin{center}
    \begin{tabular}{c||c|c|c}
	$e_{\widehat{\pi}}$ & $G_0$ & $G_1$ & $G_2$ \\
	\hline
	$\partial M$ & 1 & 1 & 2
    \end{tabular}
    \qquad
    and
    \qquad
    \begin{tabular}{c||c|c|c}
	$e_{\widehat{F}}$ & $G_0$ & $G_1$ & $G_2$ \\
	\hline
	$\partial_0 SM$ & 0 & 0 & 1
    \end{tabular},
\end{center}
where we have denoted the boundary hypersurfaces of $\partial_+ SM\times [0,1]$ by $G_{0/1} = \partial_+ SM\times \{0/1\}$ and $G_2 = \partial_0 SM\times [0,1]$. 
%%% say something about b-fibrations
Moreover, in the spirit of Gelfand transforms associated with double fibrations, one may express $I_0$ and $I_0^\sharp$ as composition of pushforwards and pullback by such b-fibrations as follows: 
\begin{align*}
	(\frac{1}{\tau} I_0 f ) d\Sigma^{2d-2} = \widehat{F}_\star \left( [\widehat{\pi}^\star f] d\Sigma^{2d-2}\ du \right), \qquad (I_0^\sharp h)\ dM_x = \widehat{\pi}_\star \left( \widehat{F}^\star (\tau\mu h) d\Sigma^{2d-2} \ du \right). 
\end{align*} 

At that point, Melrose's pushforward and pullback theorems can immediately give some versions of Theorems \ref{thm:Calphamapping} and \ref{thm:I0starIDX} above, see \cite[Sec. 3.3]{Mazzeo2021}. However, for the case of $I_0^\sharp$, it results in overestimates of the index sets produced, namely, 
\begin{align}
    I_0^\sharp (C_\alpha^{\infty}(\partial_+ SM)) \subset C^\infty(M) + \sqrt{\rho}\ C^\infty(M). 
    \label{eq:overprediction}
\end{align}
Comparing to \eqref{eq:PUI0sharp}, further work had to be done in \cite{Mazzeo2021} in order to understand why the $\sqrt{\rho}\ C^\infty$ expansion should not appear in \eqref{eq:overprediction}. To that end, the next paragraph takes a deeper dive into the proof of Theorem \ref{thm:I0starIDX}, and gives some insights into the resolution of this issue. 

\paragraph{Cancellation of poles and a sharpening of the mapping properties of $I_0^\sharp$.} 

\begin{proof}[Proof of Theorem \ref{thm:I0starIDX} (sketch)]
	To refine the mapping properties of $I_0^\sharp$, one looks for cancelling mechanisms in the proof of the PushForward theorem that uses Mellin transforms. Specifically, if the goal is to compute the index set of $I_0^\sharp g$ for some term of the form $g = a(\omega) \tau^\gamma (\log \tau)^k \chi(\tau)$ with $a\in C^\infty(\partial_0 SM)$, $(\gamma,k)\in \Cm\times \Nm_0$, the following Mellin characterization is key: the index set of $I_0^\sharp g \frac{d Vol}{\rho}$ coincides with the poles of the Mellin functional 
	\begin{align*}
		z\mapsto \left( C^\infty(\partial M)\ni h \mapsto \langle [I_0^\sharp g]_M(z), h\rangle := \int_{\partial M} h(y) \int_0^\varepsilon \rho^z I_0^\sharp g \frac{d Vol}{\rho}\right).
	\end{align*}
	Specifically, a pole of order $\ell + 1$ at $z_0\in \Cm$ of $z\mapsto \langle [I_0^\sharp g]_M(z), h\rangle$ gives rise to a term $\rho^{-z_0} (\log \rho)^\ell$ in the expansion of $I_0^\sharp g \frac{d Vol}{\rho}$ off of $\partial M$, and {\it vice versa}. Computations in \cite{Mazzeo2021}, exploiting Santal\'o's formula and \eqref{eq:bmap}, allow to rewrite the above quantity as 
	\begin{align}
		\langle [I_0^\sharp g]_M(z), h\rangle = \int_{\partial_+ SM} a(\omega) \tau^{2z+\gamma+1} \log^k \tau \chi(\tau) B(\tau,\omega; z) \frac{\mu d\Sigma^{2d-2}}{\tau^2},
		\label{eq:rewrite}
	\end{align}
	where we define the generalized\footnote{This recovers the traditional Beta function $B(z,z)$ when $h= F \equiv 1$.} Beta functional
	\begin{align*}
		B(\tau,\omega; z) := \int_0^1 h(y\circ \Upsilon (\tau,\omega,u)) F(\tau,\omega,u)^{z-1} (u (1-u))^{z-1}\ du. 
	\end{align*}
	One then derives asymptotic expansions in $\tau$ (i.e. in $\partial_+ SM$ off of $\partial_0 SM$) for the terms appearing in \eqref{eq:rewrite}, namely
	\begin{align*}
		B(\tau,\omega; z) &\sim h(\pi(\omega)) \two(\omega)^{z-1} B(z,z) + \sum_{q=1}^\infty \tau^{2q} B_{q}(\omega; z), \\
		\mu d\Sigma^{2d-2} &\sim \sum_{p=0}^\infty s_p(\omega) \tau^{2p+1} d\tau d\Omega, \qquad (\Omega:\text{ volume form on } \partial_0 SM), 
	\end{align*}
	where $s_0$ admits an explicit expression (see \cite[Lemma 5.5]{Mazzeo2021}), and where for each $q\ge 1$, $B_q(\omega; z)$ has simple poles at $-\Nm_0$, and simple zeros at $-\Nm_0 - q - 1/2$ (see \cite[Lemmas 5.3, 5.4]{Mazzeo2021}). 
	
	Combining all of this together, each term in the asymptotic expansion 
	\begin{align*}
		B(\tau,\omega; z) \frac{\mu}{d\Sigma^{2d-2}}{\tau^2} \sim \sum_{p\ge 0} \tau^{2p} C_p(\omega; z) \frac{d\tau}{\tau}\ d\Omega		
	\end{align*}
	gives rise to a contribution to $\langle [I_0^\sharp g]_M(z), h\rangle$ of the form 
	\begin{align}
		\int_{\partial_0 SM} a(\omega) C_p(\omega; z)\ d\Omega \cdot \int_0^\varepsilon \tau^{2(z+p)+\gamma+1} \log^k \tau \chi(\tau) \frac{d\tau}{\tau},
		\label{eq:product} 
	\end{align} 
	where the second factor has a pole of order $k+1$ at $z= - p - \frac{\gamma+1}{2}$, and where the first factor is seen to have simple poles at $-\Nm_0$ and {\em simple zeros at $-\Nm_0-p-1/2$}.
	
	We now briefly explain how this very last fact is what is responsible for cancellation of certain asymptotic terms (via pole cancellations at the Mellin level), and results in the narrowing of predicted index sets when mapping through $I_0^\sharp$: for example if $g\in C_\alpha^\infty (\partial_+ SM)$, then $g\sim \sum_{\ell\ge 0} \tau^{2\ell} a_\ell(\omega)$ near $\partial_0 SM$. For each term in this asymptotic expansion, the calculations above (with $\gamma = 2\ell$ and $k=0$) predict that the second factor in \eqref{eq:product} has a pole of order 1 at $z = -p - \ell - \frac{1}{2}$ while the first factor has simple poles at $-\Nm_0$ and simple zeros at $-\Nm_0 - p- \frac{1}{2}$. In particular, the product of both terms in \eqref{eq:product} does not have poles outside possibly the negative integers, which corresponds to a classical Taylor expansion off of $\partial M$. This in particular recovers \eqref{eq:PUI0sharp} which, without the presence of additional zeros, would have given \eqref{eq:overprediction}.
\end{proof}

\subsection{A non-convex example: the backprojection operator on $\Rm^2$} \label{sec:Katsevich}

While the context of the result presented below is not that of a convex manifold (rather, a complete, asympotically conic manifold), the results' formulation bear enough similarities that should be mentioned here. In \cite{Katsevich2001}, Katsevich is concerned with the mapping properties of the adjoint of the {\em Radon transform}\footnote{i.e., the transform that integrates over hyperplanes. On functions in $\Rm^2$, this coincides, up to reparameterization, with the X-ray transform.} on $\Rm^n$. The expression of this adjoint corresponds to the backprojection operator \eqref{eq:backproj} when $n=2$. Specifically, for $n\ge 2$, article \cite{Katsevich2001} is concerned with the study for the two operators
\begin{align}
	Rf(\alpha,p) = \int_{\Rm^n} f(x) \delta (\alpha\cdot x - p)\ dx, \qquad R^* \mu (x) = \int_{\Sm^{n-1}} \mu(\alpha,\alpha\cdot x)\ d\alpha,
	\label{eq:Radon}
\end{align}
where $\mu$ is defined on the space of hyperplanes $Z_n = \Sm^{n-1}_\alpha\times \Rm_p$ and satisfies the evenness condition $\mu(\alpha,p) = \mu(-\alpha,-p)$. For $n=2$, these are the analogues of $I_0$ and $I_0^\sharp$ on $\Rm^2$.  

To describe the results of \cite{Katsevich2001} in the language of the previous sections, we briefly introduce some notation: given an $\Nm_0$-valued sequence ${\bf M} = \{M_k\}_{k\in \Nm}$ we denote the set $E({\bf M}):= \cup_{k\in \Nm} \{k\}\times \{0,\dots,M_k\}$. Note that $E({\bf M})$ is a $C^\infty$-index set in the sense of Definition \ref{def:idxset} if and only if the sequence ${\bf M}$ is non-decreasing. One can think of such sets as a subclass of index sets of ``classical Taylor expansions with log terms''. We denote $\overline{\Rm^n}$ the radial compactification of $\Rm^n$, a manifold with boundary, of interior $\Rm^n$ and boundary defining function $\rho = 1/|x|$. Similarly, $\overline{Z_n} = \Sm^{n-1} \times \overline{\Rm}$ is the radial compactification of $Z_n$ with boundary defining function $1/|p|$. With such notation introduced, the results in \cite{Katsevich2001} can be described as follows: 

\smallskip
$\bullet$ Theorem 1 in \cite{Katsevich2001} states that for any sequence ${\bf M}$ as above, $R^*$ maps, injectively, functions in $\A_{\phg}^{E({\bf M})}(\overline{Z_n})$ to functions in $\A_{\phg}^{E({\bf J})}(\overline{\Rm^n})$ for some sequence ${\bf J}$ that may depend on ${\bf M}$. An inversion formula is also provided on such elements. 

\smallskip
$\bullet$ Theorem 2 in \cite{Katsevich2001} provides sharp descriptions of the images, under $R^*$, of the smaller space ${\cal S}(Z_n)$\footnote{the Schwartz space of smooth functions with rapid decay on $Z_n$, also described as the space of smooth functions on $\overline{Z_n}$ vanishing to infinite order at the boundary, denoted $\dot{C}^\infty(\overline{Z_n})$} and the even smaller space $C_c^\infty(Z_n)$. As mentioned in the beginning of Section \ref{sec:forward_backproj}, backprojection operators generally produce boundary behavior even when applied to functions supported away from the boundary. As an example of this principle, here $R^* ({\cal S}(Z_n))$ (in particular $R^* (C_c^\infty(Z_n))$) produces elements in $(1/|x|) C^\infty(\overline{\Rm^n})$. 

\smallskip
$\bullet$ In light of this last comment, for $\mu\in {\cal S}(Z_n)$, Proposition 1 in \cite{Katsevich2001} gives the explicit expression of the Taylor coefficients, off of $\partial (\overline{\Rm^n})$, of $f = R^* \mu$ in terms of integrals of $\mu$.

The proof techniques are explicit computations for the asymptotic expansions of oscillatory integrals, along with Mellin-based arguments to compute index sets and coefficients. This is similar in spirit to what is done in Section \ref{sec:bfib}.  

%\begin{theorem}[Theorem 1 in \cite{Katsevich2001}]\label{thm:Katsevich}
%	Let ${\cal X}$ be the space of $C^\infty(Z_n)$ even function that admit the expansion
%	\begin{align*}
%		\mu(\alpha,p) \sim \sum_{k=0}^\infty \sum_{m=0}^{M_k} \mu_{k,m}(\alpha) \frac{\ln^m p}{p^{k+1}}, \quad p\to +\infty,\ \mu_{k,m} \in C^\infty(\Sm^{n-1}),
%	\end{align*}
%	which is uniform with respect to $\alpha\in \Sm^{n-1}$ and can be differentiated with respect to $\alpha$ and $p$ any number of times. Let ${\cal Y}$ be the space of $C^\infty(\Rm^n)$ functiosn that admit the expansion
%	\begin{align*}
%		f(r\beta) \sim \sum_{k=0}^\infty \sum_{j=0}^{J_k} f_{k,j}(\beta) \frac{\ln^j r}{r^{k+1}},\ r\to +\infty,\ f_{k,j} \in C^\infty(\Sm^{n-1}),
%	\end{align*}
%	which is uniform with respect to $\beta\in \Sm^{n-1}$ and can be differentiated with respect to $x = r\beta$ any number of times. Then $R^* {\cal X} = {\cal Y}$, $R^* \colon {\cal X} \to {\cal Y}$ is injective, and the inversion formula $(R^*)^{-1} = \gamma_n R {\cal I}^{1-n}$ hold on ${\cal Y}$. 		
%\end{theorem}

%%%%%%%%%%%%%%%%%%%%%%%%%%%%%%%%%%%%%%%%%%%%%%%%%%%%%%%%%%%%%%%%%%%%%%%%%%%% ISOMORPHISM

\section{Isomorphism properties for normal operators on simple manifolds} \label{sec:isomorphism}

Now that settings have been given where to formulate forward mapping properties of $I_0$ and $I_0^\sharp$, it is natural to consider (weighted) compositions of both operators that make good candidates for invertible operators. We first do this in Section \ref{sec:consequences} and formulate some conjecture (Conjecture \ref{conj:isom} below), and recall the contexts in which past sections (notably Sec. \ref{sec:EuclDisk}) show that the conjecture is true on simple geodesic disks of constant curvature. In Section \ref{sec:BdM}, we then cover a general isomorphism result that is particularly adapted to the normal operator $I_0^\sharp I_0$, based on the Boutet de Monvel calculus and its generalizations to more general transmission conditions.

\subsection{Consequences of the forward mapping properties and a conjecture} \label{sec:consequences}

Some interesting consequences can be deduced from the forward mapping theorems above, specifically, Proposition \ref{prop:fwdweightedXray}, Theorems \ref{thm:Calphamapping} and \ref{thm:I0starIDX}. 

\paragraph{Consequences of Proposition \ref{prop:fwdweightedXray}.} An immediate corollary of Proposition \ref{prop:fwdweightedXray} for $I_0 := {\cal I} \circ \pi^*$ is that on $(M,g)$ a convex, nontrapping Riemannian manifold, for any $\gamma>-1$, $I_0\colon d^\gamma C^\infty(M) \to \tau^{2\gamma+1} C_{\alpha,+}^\infty(\partial_+ SM)$ is a well-defined map. Combining this with \eqref{eq:PUI0sharp}, we can get a continuous family of ``normal'' versions of $I_0$ mapping $C^\infty(M)$ into itself: 
\begin{corollary}\label{cor:gammaconj}
	Let $(M,g)$ be a convex, non-trapping Riemannian manifold, and define $\tau, d$ as above. Then for any $\gamma>-1$, the operator $I_0^\sharp \tau^{-2\gamma-1} I_0 d^\gamma$ maps $C^\infty(M)$ into itself. 
\end{corollary}

\paragraph{Consequences of Theorems \ref{thm:Calphamapping} and \ref{thm:I0starIDX}.} Further conclusions can be drawn by successively applying Theorems \ref{thm:Calphamapping} and \ref{thm:I0starIDX}:
\begin{enumerate}
    \item Under $I_0^\sharp I_0$, we get the following infinite sequence of Fr\'echet spaces mapped into one another
	\begin{align*}
	    d^{-1/2}C^\infty(M) \to C^\infty(M) \to F_1 \to F_2 \to \dots,
	\end{align*}
	where $F_k := \sum_{\ell=0}^k (\rho\log \rho)^\ell C^\infty(M)$. Moreover, one may find that such log terms genuinely appear at every step. This can be thought of as a somewhat undesireable feature of the operator $I_0^\sharp I_0$.
    \item Under $I_0^\sharp \frac{1}{\tau} I_0$, one may find that $C^\infty(M)$ gets mapped into itself, while some  other spaces can make a 2-cycle under iteration, namely
	\begin{align*}
	    C^\infty(M) + \log \rho\ C^\infty(M) \to C^\infty(M) + \rho^{1/2} C^\infty(M) \to C^\infty(M) + \log \rho\ C^\infty(M).
	\end{align*}
    \item As a generalization of Corollary \ref{cor:gammaconj}: Fixing $\bt$ an $\alpha$-bdf for $\partial_+ SM$, $\rho$ a bdf for $M$ and a real number $\gamma>-1$, one may use Theorems \ref{thm:Calphamapping} and \ref{thm:I0starIDX} to show that the operator $I_0^\sharp \bt^{-2\gamma-1} I_0 \rho^\gamma$, a bounded and self-adjoint operator on $L^2(M, \rho^\gamma\ dVol)$, maps $C^\infty(M)$ into itself. 
\end{enumerate} 

In light of the last observation, the following conjecture was formulated in \cite{Mazzeo2021}: 

\begin{conjecture}\label{conj:isom}
	Let $(M,g)$ be a simple Riemannian manifold with boundary defining function $\rho$, and fix $\gamma>-1$. Then there exists an $\alpha$-bdf $\bt$ for $\partial_+ SM$ such that the operator $I_0^\sharp \bt^{-2\gamma-1} I_0 \rho^\gamma$ is an isomorphism of $C^\infty(M)$. 
\end{conjecture}

So far, Conjecture \ref{conj:isom} is proved true in the following cases: 
\begin{itemize}
	\item On $M$ the Euclidean unit disk with $\bt = \mu = \frac{\tau}{2}$ and any $\gamma>-1$. This is Theorem \ref{thm:Cinf}. 
	\item On simple geodesic disks in constant curvature models (as in Sec. \ref{sec:CCDs} above) and any $\gamma>-1$. This is Theorem \ref{thm:CCD}, combining the previous case with the property of projective equivalence. 
	\item On any simple Riemannian manifold with boundary, for $\gamma=-1/2$ (no need to specify $\bt$ since $\bt^{-2\gamma-1}\equiv 1$ in this case). 
\end{itemize}

In the next section, we now cover the last case just mentioned, whose proof is methodologically very different from that of the previous cases.

\subsection{The isomorphism $I_0^\sharp I_0\colon d^{-1/2}C^\infty(M) \to C^\infty(M)$ via $-1/2$ transmission condition} \label{sec:BdM}

The first isomorphism result, preceding Conjecture \ref{conj:isom}, was concerning the operators $I_0^\sharp I_0$ on {\em any} simple manifold, combining the extendibility property of that operator with the theory of transmission conditions. The statement, slightly reformulated in the present notation, reads as follows.

\begin{theorem}[Theorem 4.4 in \cite{Monard2017}] \label{thm:AOS1}
    Let $(M,g)$ be a simple manifold of dimension $\ge 2$, and let $d_M$ be a boundary defining function for $M$. Then for any $s>-1$, the operator 
    \begin{align*}
	    I_0^\sharp I_0 \colon H^{-1/2 (s)}(M) \to H^{s+1} (M)
    \end{align*}
    is a homeomorphism. Moreover, $I_0^\sharp I_0\colon d_M^{-1/2} C^\infty(M)\to C^\infty(M)$ is an isomorphism. 
\end{theorem}

In the statement, $H^{s+1}(M)$ is the classical Sobolev space of order $s+1$, and $H^{-1/2(s)}(M)$ is H\"ormander's $(-1/2)$-transmission Sobolev space of order $s$. The main advantage of working with $I_0^\sharp I_0$, is that it is an {\em extendible} in the following sense: if $(\widetilde{M},\tilde{g})$ is a simple neighborhood of $M$ (with $\tilde{g}|_M = g$), and if one were to define the analogue $\widetilde{I_0^\sharp I_0}$ on $L^2(\widetilde{M})$, then one has the property that 
\begin{align}
    r_M \circ \widetilde{I_0^\sharp I_0} \circ e_M = I_0^\sharp I_0,
    \label{eq:extendibility}
\end{align}
where $e_M$ is the extension-by-zero operator, and $r_M$ is the restriction-to-$M$ operator. Such a situation is favorable to address boundary issues, in particular by translating solvability issues of the equation $I_0^\sharp I_0 u = f$ in terms of symmetry properties (a.k.a. transmission conditions) of the full symbol of (the pseudodifferential operator) $\widetilde{I_0^\sharp I_0} $ across $\partial M$. This was the original context to Boutet de Monvel's original transmission condition in \cite{Monvel1971}, providing a way to getting Fredholmness of equations of the form $Pu = f$ on a domain with boundary when $P$ was an extendible \PsiDO. Such a theory was then enriched (by, e.g., H\"ormander, Vishik-Eskin, and Grubb) to more general transmission conditions, which occur in the present case and also in the case of fractional Laplacians (see \cite{Monard2017} for more detail). 

\begin{proof}[Proof of Theorem \ref{thm:AOS1} (roadmap)]
    The proof of Theorem \ref{thm:AOS1} goes as follows: 
    \begin{itemize}
	    \item One first shows that $I_0^\sharp I_0$, in all contexts of the statement of Theorem \ref{thm:AOS1}, is a Fredholm operator. This boils down to checking that the full symbol of the extended operator $P = \widetilde{I_0^\sharp I_0}$ (a \PsiDO of order $-1$ in the interior of $\widetilde{M}$, with full classical symbol $p\sim \sum_{j=0} p_j(x,\xi)$, with $p_j$ homogeneous of degree $-1-j$ in $\xi$) satisfies a $-1/2$ transmission condition across $\partial M$ in the sense that, for every $x\in \partial M$, and $j\ge 0$, $p_j(x,\nu_x) = (-1)^j p_j(x,-\nu_x)$. In fact this equality holds even more strongly for any $\xi\in T_x \widetilde{M}$ (not just $\nu_x$), a feature referred to as the {\em evenness} of the symbol, see \cite[Lemma 4.7]{Monard2017}.
	\item One then has to show that $I_0^\sharp I_0$ has (i) trivial kernel and (ii) trivial co-kernel. Item (i) can be proved upon showing that elements in $d_M^{-1/2} C^\infty(M) \cap \ker I_0^\sharp I_0$ are in fact smooth up to the boundary so that classical injectivity results can be invoked. Smoothess up to the boundary is obtained by extending both $f$ and $I_0^\sharp I_0$ to a neighbourhood of $M$, and exploiting elliptic regularity there. Then item (ii) exploits duality arguments along with item (i).
    \end{itemize}    
\end{proof}

%%%%%%%%%%%%%%%%%%%%%%%%%%%%%%%%%%%%%%%%%%%%%%%%%%%%%%%%%%%%%%%%%%%%%%%%%%%%%%%%%% ATTENUATED

\section{Results for attenuated X-ray transforms} \label{sec:atrt}

As an example of how the results of the previous sections can help answer Q1 and Q2 for generalizations of the X-ray transform, we now treat the example of the {\em attenuated X-ray transform} (a.k.a the X-ray transform with Higgs field). Such a transform has a long history \cite{Novikov2002a,Novikov2002,Salo2011,Paternain2012,Monard2015,Assylbekov2017}, and the methods and ideas presented here appear more recently in \cite{Monard2021,Bohr2021a}. After covering preliminaries in Section \ref{sec:atrt_prelims}, Section \ref{sec:atrt_fwd} discuss recent forward mapping properties on general convex, non-trapping manifolds, while Section \ref{sec:atrt_perturb} presents all results that can be established in the case where the background geometry already benefits from the scales of Hilbert spaces \eqref{eq:SobolevZernike} and \eqref{eq:scaledata} (or \eqref{eq:scale_iso}), which already form a sharp setting for the {\em un}-attenuated X-ray transform.

\subsection{Preliminaries}\label{sec:atrt_prelims}

 Given an integer $m\in \Nm$ and an attenuation matrix (a.k.a. Higgs field) $\Phi\in C^\infty(M; \Cm^{m\times m})$, one may define the {\em attenuated X-ray transform} 
\begin{align*}
	I_\Phi \colon C_c^\infty(M^{int}, \Cm^m) \to C_c^\infty( (\partial_+ SM)^{int}, \Cm^m)
\end{align*}
given by $I_\Phi f:= u|_{\partial_+ SM}$, where $u \colon SM\to \Cm^m$ solves the transport problem 
\begin{align}
	Xu + \Phi u = -f \quad (SM), \qquad u|_{\partial_- SM} = 0.
	\label{eq:utransport}
\end{align}
There is a convenient way to factor $I_\Phi$ in terms of the unattenuated X-ray transform \eqref{eq:I}. Namely, one may construct (see, e.g., \cite[Sec. 5.2.2]{Monard2019}) a smooth integrating factor for $\Phi$, i.e. a smooth solution $R_\Phi\colon SM\to GL(m,\Cm)$ to the transport equation $X R_\Phi + \Phi R_\Phi = 0$ on $SM$. In terms of $R_\Phi$, the transform $I_\Phi$ can be written as follows:
\begin{align}
	I_\Phi f (\x,v) := R_\Phi \int_0^{\tau(\x,v)} R_\Phi^{-1}(\varphi_t(\x,v)) f(\pi (\varphi_t(\x,v)))\ dt = R_\Phi\cdot I \left( R_\Phi^{-1} f \right)(\x,v). 
	\label{eq:atXrt}
\end{align}
With this factorization, this operator can be made bounded in the setting $L^2(SM, \Cm^m) \to L^2(\partial_+ SM, \frac{\mu}{\tau}\ d\Sigma^{2n-2})$, with adjoint 
\begin{align}
	I_\Phi^* = \pi_\star R_\Phi^{-*} F^* R_\Phi^* \frac{1}{\tau},
	\label{eq:atXrtstar}
\end{align}
where $F^*$ is ``pullback by $F$'' and $R_\Phi^*$ is the pointwise Hermitian adjoint of $R_\Phi(\x,v)$. 

\subsection{Forward mapping properties of $I_\Phi$ on convex, non-trapping manifolds}\label{sec:atrt_fwd}

The following forward mapping properties of $I_\Phi^* I_\Phi$ can be rather easily derived from Proposition \ref{prop:fwdweightedXray}.

\begin{theorem}[Theorem 6.2 in \cite{Monard2021}]\label{thm:MNP_fwd} 
	Let $(M,g)$ be a non-trapping manifold with strictly convex boundary and let $\Phi \in C^\infty(M, \Cm^{m\times m})$. The operator $I_\Phi^* I_\Phi$ maps $C^\infty(M, \Cm^m)$ into itself. 
\end{theorem}

\begin{proof} The proof follows from the factorizations \eqref{eq:atXrt} and \eqref{eq:atXrtstar} and the mapping properties of each operator, also noting that $F^*$ maps $C_{\alpha,+}^\infty(\partial_+ SM)$ into $C^\infty(SM)$. 	
\end{proof}

In light of Theorem \ref{thm:MNP_fwd}, one may ask what contexts, on $(M,g)$ (e.g., dimension, geometric restrictions such as simplicity) and $\Phi$ (e.g., regularity), may suffice to make $I_\Phi^* I_\Phi$ an {\em isomorphism} of $C^\infty(M, \Cm^m)$.

%formulate the following conjecture. 
%\begin{conjecture}
%	Let $(M,g)$ be a convex, non-trapping manifold with strictly convex boundary and let $\Phi \in C^\infty(M, \Cm^{m\times m})$. Then the operator $I_\Phi^* I_\Phi$ is an isomorphism of $C^\infty(M, \Cm^m)$. 
%\end{conjecture}

%The case where $(M,g)$ is simple is more believable to be true, while cases with conjugate points may be more challenging: in the simple case, one may in fact expect that there is a Sobolev scale similar to \eqref{eq:SobolevZernike}

\subsection{Refinements on the Euclidean disk} \label{sec:atrt_perturb}

On a general Riemannian manifold, the scales \eqref{eq:SobolevZernike}, \eqref{eq:scaledata} and \eqref{eq:scale_iso} have no analogue at the moment. However, on the Euclidean disk, they serve as an appropriate functional setting to capture many properties of attenuated X-ray transforms. In particular, they allow to obtain forward Sobolev mapping properties for $I_\Phi$ (covered in Sec. \ref{sec:fwdIPhi}), isomorphism properties for $I_\Phi^* I_\Phi$ (covered in Sec. \ref{sec:NPhi}), and sharp stability estimates for $I_\Phi$ (covered in Sec. \ref{sec:stabIPhi}).

\begin{remark}
	It is expected that all schemes of proof below should also work on constant curvature disks {\it mutatis mutandis}, since those manifolds also have similar Sobolev scales as \eqref{eq:SobolevZernike}, \eqref{eq:scaledata} and \eqref{eq:scale_iso}. These are not recorded in the literature however.	
\end{remark}

\subsubsection{Forward mapping properties of $I_\Phi$ in Hilbert scales}\label{sec:fwdIPhi}

We begin with a result in \cite{Bohr2021a}, establishing forward mapping estimates for attenuated X-ray transforms on the Euclidean disk in the Sobolev scales \eqref{eq:SobolevZernike} and \eqref{eq:scale_iso}, both extended to $\Cm^m$-valued functions in the obvious way, for instance for the scale \eqref{eq:SobolevZernike}: 
\begin{align*}
	\|f\|_{\wtH^s(\Dm,\Cm^m)}^2 = \sum_{j=1}^m \|f_j\|_{\wtH^{s}(\Dm)}^2, \qquad f\colon \Dm\to \Cm^m,\qquad f = (f_1,\dots,f_m).
\end{align*}

\begin{theorem}[Theorem 4.1 in \cite{Bohr2021a}]\label{thm:BN_forward}
    Let $s\ge 0$ and suppose $\Phi$ is of regularity $C^k(\Dm)$, where $k = 2\lceil s/2 \rceil$ (the smallest even integer $\ge s$). Then for all $f\in \wtH^s(\Dm,\Cm^m)$, we have
    \begin{align*}
	\|I_\Phi f\|_{H^s(\partial_+ S\Dm)} \le C\cdot \left(1+\|\Phi\|_{C^k(\Dm)}\right)^k \cdot \|f\|_{\wtH^s(\Dm)}
    \end{align*}
    for a constant $C = C(s) >0$. For $s=1$ one may also take $k=1$.
\end{theorem}

\begin{remark}
	In the case where the attenuation is zero, it was seen in Section \ref{sec:EuclDisk} that such a mapping property can be strengthened to $I_0 \colon \wtH^s(\Dm)\to H^{s+\frac{1}{2}} (\partial_+ S\Dm)$. At present it is unclear how to upgrade Theorem \ref{thm:BN_forward} to a $\wtH^s(\Dm) \to H^{s+\frac{1}{2}} (\partial_+ S\Dm)$ forward estimate for general $\Phi$. 	
\end{remark}

\begin{proof}[Roadmap of proof of Theorem \ref{thm:BN_forward}]
	The proof is done by factoring $I_\Phi f$ as in \eqref{eq:atXrt}, namely, $I_\Phi f = R_\Phi I (R_\Phi^{-1} f)$. From there, the proof follows by combining the two key ingredients: 
	\begin{itemize}
		\item When $\Phi$ is regular enough, then $R_\Phi$ (resp. $R_\Phi|_{\partial_+ S\Dm}$) has enough regularity to be a multiplier of $\wtH^s(S\Dm)$ (resp. $H^s(\partial_+ S\Dm)$); see \cite[Lemma 4.7]{Bohr2021a}.
		\item Using Zernike polynomials and the SVD of the X-ray transform $I\colon L^2(S\Dm)\to L^2 (\partial_+ S\Dm)$, one can establish the forward mapping estimate $\|Iu\|_{H^s(\partial_+ S\Dm)}\lesssim_s \|u\|_{\wtH^s(S\Dm)}$ for all $s\ge 0$ and $u\in \wtH^s(S\Dm)$; see \cite[Lemma 4.6]{Bohr2021a}. 
	\end{itemize}
	%\Fnote{In the paper, $L^2(\partial_+ SM)$ is referred to as $L^2_\lambda$}
\end{proof}

\subsubsection{Isomorphism properties of $I_\Phi^* I_\Phi$} \label{sec:NPhi}

%The results above, although not yet portable to the general case of a simple surface (or higher-dimensional manifolds), can adapt to the perturbative setting of attenuated X-ray transforms when the background geometry is that of the Euclidean disk, as can be seen below (the results presented, although not written in the literature for constant curvature disks, generalize straightforwardly to that situation as well).

In this section, let us denote the normal operator $N_\Phi := I_\Phi^* I_\Phi$ for short. Turning Theorem \ref{thm:MNP_fwd} into an isomorphism can be done upon further assuming that $(M,g)$ is the Euclidean unit disk, and $\Theta$ is smooth, compactly supported in $\Dm^{int}$, and skew-hermitian. 

\begin{theorem}[Theorem 6.4 in \cite{Monard2021}] \label{thm:MNP_isom}
	Suppose $(M,g) = (\Dm,e)$, and let $\Phi \in C_c^\infty(\Dm^{int}, \Cm^{m\times m})$ be skew-hermitian. Then the map 
	\begin{align*}
		N_\Phi := I_\Phi^* I_\Phi \colon C^\infty(M,\Cm^m) \to C^\infty(M,\Cm^m)
	\end{align*}
	is an isomorphism. 	
\end{theorem}

\begin{proof}[Proof of Theorem \ref{thm:MNP_isom} (roadmap)]
	Since $\cap_{k\in \Nm} \wtH^k(\Dm, \Cm^m) = C^\infty(\Dm, \Cm^m)$ by \cite[Theorem 12]{Monard2019a}, Theorem \ref{thm:MNP_isom} follows upon establishing the more precise statement that, for every $k\in \Nm_0$, 
	\begin{align}
		N_\Phi \colon \wtH^{k}(\Dm, \Cm^m) \to \wtH^{k+1} (\Dm,\Cm^m)
		\label{eq:isomATRT}
	\end{align}
	is an isomorphism; see \cite[Theorem 6.18]{Monard2021}. This is done upon showing that on the Zernike Sobolev scale $\{\wtH^{k}(\Dm,\Cm^m)\}_{k\ge 0}$, the operator $K_\Phi := I_\Phi^* I_\Phi - (I^*_0 I_0) id_{m\times m}$ is a relatively compact perturbation of $(I^*_0 I_0) id_{m\times m}$. To prove this, one must look at the singularities of their Schwartz kernels near the diagonal $\{(\x,\x): \x\in \Dm^{int}\}$, and near the corner $\partial \Dm \times \partial \Dm$ of $\Dm\times \Dm$, see \cite[Sec. 6.4]{Monard2021} for details. Briefly:
	\begin{itemize}
		\item Near the diagonal, the assumption that $\Phi$ is skew-hermitian implies that both (pseudo-differential) operators $I_\Phi^* I_\Phi$ and $(I^*_0 I_0) id_{m\times m}$ have same principal symbol of order $-1$, and as such their difference is a PsiDO of order at most $-2$ while $(I^*_0 I_0) id_{m\times m}$ is of order $-1$.  
		\item Near the corner, one finds that the assumption that $\Phi$ is compactly supported in $\Dm^{int}$ implies that the Schwartz kernel of $K_\Theta$ {\em vanishes identically} in a neighborhood of $\partial\Dm\times \partial\Dm$.
		\item Using a partition of unity on $\Dm\times \Dm$, this helps chop $K_\Phi$ into pieces all of which can be made $\wtH^{k}\to \wtH^{k+2}$ bounded, and hence compact relative to the bounded operator $I_0^*I_0\colon \wtH^k\to \wtH^{k+1}$. That this is true for the 'interior pieces' comes from the first bullet above and the fact that the classical Sobolev and Zernike-Zobolev topologies are equivalent on distributions with fixed compact support in $\Dm^{int}$. See \cite[Lemmas 6.15 and 6.17]{Monard2021} for further details.
	\end{itemize}
	The above steps then imply that $I_\Phi^* I_\Phi\colon \wtH^{k}(\Dm, \Cm^m) \to \wtH^{k+1} (\Dm,\Cm^m)$ is Fredholm. Obtaining isomorphism property  \eqref{eq:isomATRT} then follows from exploiting the previously known injectivity result of attenuated transforms on simple surfaces \cite{Paternain2012}, and dualizing to produce surjectivity via a careful use of the Zernike scale $\{\wtH^{s}(\Dm, \Cm^m)\}_{s\in \Rm}$. 
\end{proof}

A refinement of Theorem \ref{thm:MNP_isom} by Bohr and Nickl in \cite{Bohr2021a} consists in reducing the regularity requirement on $\Phi$ such that $N_\Phi \colon L^2(\Dm,\Cm^m)\to \wtH^1(\Dm,\Cm^m)$ remains an isomorphism, and providing more refined information about the continuity properties of the maps $\Phi\mapsto N_\Phi$ and $\Phi\mapsto N_\Phi^{-1}$. In what follows, given $H_1, H_2$ two Hilbert spaces, we write $\B(H_1,H_2)$ for the space of continuous linear operators between them, equipped with the operator norm; $\B^\times(H_1,H_2)$ is the open subset of its invertible elements. 

\begin{theorem}[Theorem 4.8 in \cite{Bohr2021a}]\label{thm:BNisom} Let $K\subset \Dm$ be compact. 

	(i) For $\Phi\in C_K^2(\Dm,\mathfrak{u}(m))$ the normal operator $N_\Phi$ maps $L^2(\Dm,\Cm^m)$ into $\wtH^1(\Dm,\Cm^m)$. The following map is Lipschitz continuous on bounded sets: 
	\begin{align*}
		\Phi\mapsto N_\phi, \quad \text{ as map }\ C_K^2(\Dm, \mathfrak{u}(m))\to {\cal B} (L^2(\Dm),\wtH^1(\Dm)). 
	\end{align*} 

	(ii) For $\Phi\in C_K^4(\Dm, \mathfrak{u}(m))$ the operator $N_\Phi\colon L^2(\Dm, \Cm^m) \to \wtH^1(\Dm,\Cm^m)$ is an isomorphism. Moreover, the following map is continuous
	\begin{align*}
		\Phi\mapsto N_\phi^{-1}, \quad \text{ as map }\ C_K^4(\Dm, \mathfrak{u}(m))\to {\cal B}^\times (\wtH^1(\Dm), L^2(\Dm)). 
	\end{align*} 
\end{theorem}

We briefly recall the scheme of proof here, though the proof is involved and well-written in \cite{Bohr2021a}, an interested reader is encouraged to refer to the original article for the full detail.  

\begin{proof}[Proof of Theorem \ref{thm:BNisom} (roadmap)] 
	This is done by carefully proving analogues of \cite[Lemmas 6.17 and 6.18]{Monard2021} when $\Phi$ has finite regularity ($C^2$ or $C^4$). Microlocal arguments in the proof of Theorem \ref{thm:MNP_isom} are then replaced by Calder\'on-Zygmund type arguments, where the operators $N_\Phi$ and $K_\Phi$ (appearing in the proof of Theorem \ref{thm:MNP_isom} above) admit local singular integral representations in terms of some smooth kernels. Norm bounds for these operators can be derived, only requiring finitely many derivatives of said kernels, and hence of $\Phi$. Using the localization procedure in the proof of Theorem \ref{thm:MNP_isom} above, local Lipschitz estimates for the maps $C_K^2\ni \Phi\mapsto N_\Phi$ and $C_K^4\ni \Phi\mapsto K_\Phi$ are derived. These estimates, together with an approximation argument, give the analogue of \cite[Lemmas 6.17]{Monard2021}. The analogue of \cite[Lemma 6.18]{Monard2021} uses in addition the stability estimate from \cite[Theorem 5.3]{Monard2019}. 	
\end{proof}

%\begin{proof}[Proof of Theorem \ref{thm:BNisom} (roadmap)]
%	In the present case where $\Phi$ has finite regularity, this is a key contribution of \cite{Bohr2021a} as it requires a careful study of the kernel of $N_\Phi$, near the diagonal using Calder\'on-Zygmund type arguments, and near the boundary of $\Dm\times \Dm$ using a localization procedure similar to that given in the outline of Theorem \ref{thm:MNP_isom} presented above. The main result needed is formulated in \cite[Theorem 4.8.(ii)]{Bohr2021a}, stating that if $\Phi \in C_K^4(\Dm)$ (i.e. $C^4$ with fixed compact support $K\subset \Dm^{int}$), then $N_\Phi\colon L^2(\Dm)\to \wtH^1(\Dm)$ is an isomorphism and the map $C_K^4(\Dm)\ni \Phi\mapsto N_\Phi^{-1} \in {\cal B}^{\times} (\wtH^{1}, L^2)$ is continuous, ensuring that the coercivity constant $\underline{n}(\Phi) := \|N_\Phi^{-1}\|_{ {\cal B} (\wtH^{1}, L^2)}$ depends continuously on $\Phi$.	
%\end{proof}

\subsubsection{A sharp $H^{-1/2}\to L^2$ stability estimates for $I_\Phi$}\label{sec:stabIPhi}

As a follow-up to Theorem \ref{thm:BNisom} above, Bohr and Nickl formulate a sharp stability estimate for the attenuated X-ray transform on the Euclidean disk. This estimate captures the $1/2$-smoothing property of $I_\Phi$ on the Zernike-Sobolev scale \eqref{eq:SobolevZernike}.

\begin{theorem}[Theorem 4.1 in \cite{Bohr2021a}]\label{thm:BN_stab}
    Suppose $\Phi\colon \Dm\to \mathfrak{so}(m)$ is of regularity $C^4(\Dm)$ and its support is contained in a compact set $K\subset \Dm^{int}$. Then 
    \begin{align}
	\|f\|^2_{\wtH^{-1/2}(\Dm)} \le C(\Phi,K) \cdot \|I_\Phi f\|^2_{L^2(\partial_+ S\Dm)}
	\label{eq:BN_stab}
    \end{align} 
    for all $f\in L^2(\Dm,\Cm^m)$ and a constant $C(\Phi,K)>0$ which (for fixed $K$) depends continuously on $\Phi$ in the $C^4$-topology\footnote{Note: the space $L^2(\partial_+ S\Dm)$ is referred to as $L_\lambda^2(\partial_+ S\Dm)$ in \cite{Bohr2021a}.}.   
\end{theorem}

\begin{proof}[Proof of Theorem \ref{thm:BN_stab}] In light of Theorem \ref{thm:BNisom}, the self-adjoint, nonnegative operator $N_\Phi = I_\Phi^* I_\Phi \colon L^2(\Dm)\to \wtH^1(\Dm)$, has a bounded inverse, and by Theorem \ref{thm:BNisom}.(ii), the coercivity constant $\underline{n}(\Phi) := \|N_\Phi^{-1}\|_{ {\cal B} (\wtH^{1}, L^2)}$ depends continuously on $\Phi$. 

	Then \eqref{eq:BN_stab} is obtained by functional-analytic arguments. Upon denoting $B$ the unit ball in $\wtH^{1/2}$ and using that $\wtH^{-1/2} = (\wtH^{1/2})^*$, \eqref{eq:BN_stab} goes as follows: 
	\begin{align*}
		\|f\|_{\wtH^{-1/2}(\Dm)} = \sup_{h\in B} \langle f, h\rangle_{L^2} = \sup_{h\in B} \langle N_\Phi^{1/2} f, N_\Phi^{-1/2} h\rangle_{L^2} &\le \sup_{h\in B} \|N_\Phi^{-1/2} h\|_{L^2}\cdot \langle N_\Phi f, f\rangle_{L^2}^{1/2}  \\
		&\le \sqrt{\underline{n}} \|I_\Phi f\|_{L^2(\partial_+ S\Dm)},
	\end{align*}
	where the last step is obtained by operator monotonicity of the squareroot, see \cite[Lemma 4.9]{Bohr2021a} for details.
\end{proof}

%Krishnan-Sharafutdinov (Reshetnyak formulas) - \cite{Sharafutdinov2016}

\section{Perspectives} \label{sec:perspectives}

The author attempted to convey that the quest for non-standard (i.e., either anisotropic in the interior, or encapsulating non-trivial boundary behavior) Sobolev scales that accurately capture the mapping properties of X-ray transforms is an actively developing story. There is plenty of room for improvement, be it at the level of general simple Riemannian manifolds, or in potential generalizations of these results (to transforms that integrate over not-necessarily-geodesic flows, or to weighted transforms). This is in this spirit that we formulate some open questions below.

\begin{enumerate}
	\item On general simple Riemannian surfaces (or higher-dimensional manifolds), how to construct analogues of the scales \eqref{eq:SobolevZernike}, \eqref{eq:scaledata} and \eqref{eq:scale_iso}, and how to prove forward mapping properties and stability estimates for $I_0$ on them?
	\item How to generalize the present results to different flows, in particular, non-reversible ones (e.g., magnetic flows)?
	\item In light of the undesirable feature of $I_0^\sharp I_0$ pointed out in Section \ref{sec:consequences}: On a general simple manifold, can one describe the space $I_0^\sharp I_0 (H^s(M))$? what is the boundary regularity of the eigenfunctions of $I_0^\sharp I_0$ ? Are there examples where such eigenfunctions are explicitly computable?
	\item To what extent the results presented extend to tensor fields? Although some results mentioned (e.g. Theorems \ref{thm:ASstab} and \ref{thm:PSstab}) have analogues for tensor fields, a systematic treatment remains to be written. 
	\item To what extent the results extend to transforms with attenuation and/or connection? How can one push the results of Section \ref{sec:atrt_perturb} past the confines of constant curvature disks and skew-hermitian attenuations that are supported away from the boundary ?
\end{enumerate}

More speculatively: 
\begin{enumerate}
	\item How to generalize existing mapping results to non-abelian X-ray transforms (see \cite{Monard2019,Monard2021}, and transforms which integrates over other families of submanifolds (e.g., of dimension $> 1$).
	\item How to classify simple surfaces where explicit functional relations between the X-ray transform and degenerate elliptic operators such as \eqref{eq:funcRel} can be found ?
\end{enumerate}

\section*{Acknowledgement}

The author acknowledges partial support from NSF CAREER grant DMS-1943580. The author would like to thank Gabriel Paternain, Richard Nickl, Rafe Mazzeo, Joey Zou, Rohit Mishra, Plamen Stefanov and Gunther Uhlmann for the fruitful discussions and collaborations that have all heavily contributed to furthering the understanding of the topics of this review article.

\bibliographystyle{plainnat}
\bibliography{../bibliography}

\end{document}